\documentclass[reqno,11pt]{amsart}
\usepackage{amsmath,amssymb,graphics,epsfig,color,cite}

\newtheorem{theorem}{Theorem}[section]
\newtheorem{proposition}[theorem]{Proposition}
\newtheorem{lemma}[theorem]{Lemma}

\theoremstyle{remark}
\newtheorem{remark}[theorem]{Remark}
\theoremstyle{definition}
\newtheorem{definition}[theorem]{Definition}

\numberwithin{equation}{section}
\numberwithin{theorem}{section}

\newcommand{\mc}[1]{{\mathcal #1}}
\newcommand{\bb}[1]{{\mathbb #1}}
\newcommand{\const}{(\mathrm{const})} 
\newcommand{\diag}{\mathrm{diag}}
\newcommand{\rme}{\mathrm{e}}
\newcommand{\rmi}{\mathrm{i}}

\newcommand{\bigo}{\mathcal O}

\newcommand{\epsyxstar}{\eps y x^*}

\newcommand{\eps}{\varepsilon}
\newcommand{\la}{\lambda}
\newcommand{\Lameps}{\Lambda_\varepsilon}
\newcommand{\aleps}{\alpha_\varepsilon}
\newcommand{\rhoeps}{\rho_\varepsilon}

\newcommand{\goes}{\rightarrow}
\newcommand{\lan}{\left\langle}
\newcommand{\ran}{\right\rangle}

\renewcommand{\i}{\mathrm{i}}
\begin{document}

\title[Structured pseudospectral measures of a Toeplitz matrix]
{Computing the structured pseudospectrum of a Toeplitz matrix and its extreme points}

\author [P.\ Butt\`a] {Paolo Butt\`a}
\address{Paolo Butt\`a, Dipartimento di Matematica, SAPIENZA Universit\`a di Roma, P.le Aldo Moro 5, 00185 Roma, Italy}
\email{butta@mat.uniroma1.it}

\author [N.\ Guglielmi] {Nicola Guglielmi}
\address{Nicola Guglielmi, Dipartimento di Matematica Pura ed Applicata, Universit\`a degli Studi di L'Aquila, Via Vetoio - Loc. Coppito, 67010 L'Aquila, Italy} 
\email{guglielm@univaq.it} 

\author [S.\ Noschese] {Silvia Noschese}
\address{Silvia Noschese, Dipartimento di Matematica, SAPIENZA Universit\`a di Ro\-ma, P.le Aldo Moro 5, 00185 Roma, Italy}
\email{noschese@mat.uniroma1.it}



\begin{abstract}
The computation of the structured pseudospectral abscissa and radius (with respect to the
Frobenius norm) of a Toeplitz matrix is discussed and two algorithms based on a low rank 
property to construct extremal perturbations are presented.
The algorithms are inspired by those considered in \cite{GO11} for the unstructured
case, but their extension to structured pseudospectra and analysis presents several difficulties. 
Natural generalizations of the algorithms, allowing to draw significant sections of the structured 
pseudospectra in proximity of extremal points are also discussed. Since no algorithms are 
available in the literature to draw such structured pseudospectra, the approach we present
seems promising to extend existing software tools (Eigtool \cite{Wri02}, Seigtool \cite{KKK10}) to structured
pseudospectra representation for Toeplitz matrices. 
We discuss local convergence properties of the algorithms and show some applications to
a few illustrative examples. 
\end{abstract}

\maketitle
\thispagestyle{empty}

\section{Introduction}
\label{sec:1}

There is a growing development of structure-preserving algorithms for  structured problems. Toeplitz matrices arise in many applications, including the solution of ordinary  differential equations, whence  it is meaningful to investigate  the sensitivity of the eigenvalues of a Toeplitz matrix with respect to finite structure-preserving perturbations and,  mainly,  the sensitivity of  the rightmost eigenvalue. The structure is given by the location of the nonzero diagonals of the matrix.

We add the structure requirement to the classical definition of $\eps$-pseudo\-spec\-trum; see, e.g., \cite{TE}. Given $\eps>0$, the  structured $\eps$-pseudospectrum  of a given Toeplitz matrix $A \in \mathbb{C}^{n \times n}$  is the set of all eigenvalues of $A+\eps E$ for some Toeplitz matrix $E \in \mathbb{C}^{n \times n}$ with unitary norm,  and with the same sparsity structure as $A$.  As an example, if $A$ is a tridiagonal Toeplitz matrix, we consider all tridiagonal Toeplitz perturbation matrices of norm equal to $\eps$. The structured pseudospectral abscissa is the maximal real part of points in the structured pseudospectrum. 

We remark that the notion of $\eps$-pseudospectrum depends on the choice of the matrix norm. In literature, the spectral norm has been largely used also in the  structured case\cite{BGK01,G06,R06}. Guglielmi and Overton presented in \cite{GO11} an efficient algorithm for computing the pseudospectral abscissa in the spectral norm.  For our purposes, the Frobenius norm turns out to be the most appropriate. Since the points in a structured pseudospectrum are exact eigenvalues of some nearby  Toeplitz matrix with the same structure diagonals  as $A$, we are in a position to use results from the  literature concerning the eigenvalue sensitivity to machine perturbations, that is to say infinitely small structured perturbations. The structured condition numbers of an eigenvalue $\lambda \in A$ is indeed a first-order measure of the worst-case effect  on $\lambda$ of perturbations of the same structure as  $A$. The structured conditioning measures we deal with can be computed endowing the subspace of matrices  with the Frobenius norm; see, e.g., \cite{HH,KKT,NP} and references therein.

Here we are concerned with the computation of the rightmost points in the structured pseudo\-spectrum  of a Toeplitz matrix. Since we are limiting finite perturbations to a given Toeplitz structure, it is not surprising that the main difference in our extension of  the algorithm in \cite{GO11} consists in replacing the classical eigenprojection for a simple eigenvalue with its structured analogue (normalized in the Frobenius norm).

We remark that we may generalize the above statements  to non-real Hankel matrices, considering antidiagonals in place of diagonals. Similarly, other symmetry-pattern nonnormal matrices can be treated (a symmetry-pattern being a structure that exhibits a kind of symmetry, like reflection or translation  \cite{NP}); for instance, general persymmetric, skew-persymmetric, complex symmetric or complex skew-symmetric matrices. In all cases, matrix perturbations with the given sparsity and symmetry-pattern have to be considered.

In this paper we thoroughly investigate the tridiagonal Toeplitz structure. The motivation is that the eigenvalues and eigenvectors of tridiagonal Toeplitz matrices are known in closed form, and all  ingredients of our analysis are easily computable \cite{NPR}.  Additionally, it is well known that the boundary of the $\eps$-pseudospectrum in spectral norm of a tridiagonal Toeplitz matrix approximates an ellipse, as $\varepsilon$ approaches zero and the dimension $n$ goes to infinity \cite{RT}, and  a slightly modified version of the algorithm in \cite{GO11},  which succeeds in plotting the boundary of the $\eps$-pseudospectrum, designs  in fact  an ellipse in the tridiagonal Toeplitz case.  Analogously, we adapt the new algorithm in order to investigate the boundary of the structured $\eps$-pseudospectrum  in the Frobenius norm.

The paper is organized as follows. In Section \ref{sec:2} we define the algorithm and show how to modify it to compute also the pseudospectral radius, and partially draw the pseudospectral boundary. In Section \ref{sec:3} we characterize the fixed points of the algorithm in the tridiagonal case. In Section \ref{sec:4} we derive a local convergence analysis, establishing that the algorithm is linearly convergent to local maximizers of the structured pseudospectral abscissa. Finally, in Section \ref{sec:5} the algorithms are tested on some examples.

\section{The algorithm}
\label{sec:2}

We start with some notation and definition. Given a Toeplitz matrix $A\in \bb C^{n\times n}$ we denote by $\mc S$ the subspace of all Toeplitz matrices in $\bb C^{n\times n}$ with same sparsity structure as $A$. We denote by $M|_\mc S$ the matrix in $\mc S$ closest to $M\in \bb C^{n\times n}$ with respect to the Frobenius norm. It is straightforward to verify that $M|_\mc S$ is obtained by replacing in each structure diagonal all the entries of $M$ with their arithmetic mean. We also define the normalized projection,
where $\|\cdot\|_F$ stands for the Frobenius norm,
\begin{equation*}
M|_\mc T := \frac{M|_\mc S}{\|M|_\mc S\|_F}\;.
\end{equation*}

If $\lambda$ is a simple eigenvalue of a matrix $M\in \bb C^{n\times n}$, a corresponding pair of right and left eigenvectors  $x$ and $y$ are said normalized to be RP-compatible if $\|x\|_2 = \|y\|_2=1$ and $y^*x$ is real and positive. 
\begin{lemma}[see \cite{NP}]
\label{lem:p1}
Let $\lambda$ be a simple eigenvalue of a Toeplitz matrix $A$ with corresponding right and left eigenvectors $x$ and $y$ normalized to be RP-compatible. Given any Toeplitz matrix $E$ with $\|E\|_F = 1$, let $\lambda_E(t)$ be an eigenvalue of $A+tE$ converging to $\lambda$ as $t\to 0$. Then,
$$
|\dot\lambda_E(0)| \le \max\left\{\left|\frac{y^*Gx}{y^*x}\right| , \; \|G\|_F = 1,\, G \in \mc S\,\right\} = \frac{\|yx^*|_{\mc S}\|_F}{y^*x}
$$
and
\begin{eqnarray}
\label{f2}
\nonumber
\dot\lambda_E(0) & = & \frac{\|yx^*|_{\mc S}\|_F}{y^*x} >0\qquad \mathrm{if} \qquad E=yx^*|_{\mc T}\;.
\end{eqnarray}
\end{lemma}
\begin{remark}
If the right and left eigenvectors are normalized so that $\|x\|_2 = \|y\|_2 = 1$ and arg$(y^*x) = - \theta$ then arg$(\dot\lambda_E(0) )= \theta$ if $E=yx^*|_{\mc T}$.  Indeed, since $y^*(yx^*)|_{\mc S}x = \|yx^*|_{\mc S}\|_F^2$ (see \cite[Lemma 3.2]{NP}),
$$
\dot \lambda_E(0) = \frac{y^*(yx^*)|_{\mc S}x}{\|yx^*|_{\mc S}\|_F} \frac{1}{y^*x}= \frac{\|yx^*|_{\mc S}\|_F}{y^*x}\;.
$$
\end{remark}
Lemma \ref{lem:p1} allows to extend the algorithm introduced in \cite{GO11} to the case of Toeplitz structure.

\subsection{Pseudospectral abscissa}

We define  
\begin{eqnarray}
\aleps^{\mc T} (A) & = & \max\{\Re(\lambda)\colon\lambda \in \Lameps^{\mc T}(A) \}\;, 
\label{eq:psa}
\nonumber
\end{eqnarray}
the structured pseudospectral abscissa, where 
\[
\Lameps^{\mc T}(A) = \left\{ \lambda \in \bb C \colon \lambda \in \Lambda(A+E) \quad \mbox{with} \ E \in {\mc S}, \| E \|_F \le \eps  \right\}\;.
\]

The following algorithm allows to compute locally rightmost points of the $\eps$-pseudospectrum.

\noindent{\bf Algorithm 1.} Let  $\lambda_0$ be a rightmost eigenvalue of a given Toeplitz matrix $A\in \bb C^{n\times n}$ with corresponding right and left eigenvectors  $x_0$ and $y_0$ normalized to be RP-compatible. Set $B_1 = A+\eps\, y_0x_0^*|_\mc T$. 

For $k=1,2,\ldots$, let $\lambda_k$ be a rightmost eigenvalue of $B_k$ closest to $\lambda_{k-1}$. Let $x_k$ and $y_k$ be corresponding right and left eigenvectors normalized to be RP-compatible. Set $B_{k+1} = A+\eps\, y_kx_k^*|_\mc T$.  $\diamond$

\medskip

We denote by $M_\eps$ the iteration map associated to Algorithm 1, i.e.
\begin{equation*}
y_{k-1} x_{k-1}^*\big|_\mc T \quad \overset{M_\eps}{\longrightarrow} \quad y_k x_k^*\big|_\mc T\;.
\end{equation*}
By the definition of the algorithm it follows immediately that the fixed points of $M_\eps$ are given by the pairs $(x,y)$ solution to
\begin{equation}
\label{p1}
\left\{\begin{array}{l}
y^* \left(A+\eps\, yx^*|_\mc T \right)= \lambda y^*\;, \\[0.2cm] 
\left(A+\eps\, yx^*|_\mc T \right) x = \lambda x\;.
\end{array}\right.
\end{equation}

\subsection{Local maxima and stationary points of Algorithm 1}

We are now interested to relate locally rightmost points of the $\eps$-pseudo\-spectrum
to stationary points of our algorithm, that is fixed points of the map $M_\eps$.
Let
\[
{\mc M} = \{ E \in \bb C^{n\times n} : E \ \in {\mc S}, \ \| E \|_F = 1 \}
\]
and consider the differential equation 
\begin{equation}
\left\{ \begin{array}{l}
\dot{E}(t) = y(t)\,x(t)^* \big|_{\mc T} -\lan E(t), y(t)\,x(t)^* \big|_{\mc T} \ran\, E(t)
\\[0.2cm]
E(0) = E_0 \in \mc M
\end{array}
\right.
\label{eq:ode}
\end{equation}
where $\lan \cdot,\cdot \ran$ denotes the Frobenius inner product, i.e., $\langle E,F \rangle = {\rm trace} (E^* F)$,
and $y(t), x(t)$ are respectively the left and right eigenvectors associated to the rightmost eigenvalue of $A+\eps E(t)$, which we assume to be simple, normalized such that $y(t)^* x(t) > 0$  (this is similar to the differential equation analyzed in \cite{GL11,GL12} in the case of standard complex and real pseudospectra). 

We easily observe that $\dot{E} \in T_E(\mc M)$ (the tangent hyperplane to ${\mc M}$ at $E$) 
which implies that $E(t) \in \mc M$ for all $t$. In fact, $\dot E(t) \in {\mc S}$ and 
$\langle E(t), \dot{E}(t) \rangle = 0$.
\begin{lemma}[Equilibria]
\label{lem:equil}
If $x$ and $y$ fulfil \eqref{p1} with $y^* x \ne 0$, then $E = y x^* \big|_{\mc T}$
is an equilibrium of \eqref{eq:ode}. Viceversa, if $E$ is an equilibrium of {\rm (\ref{eq:ode})} then $E = y x^* \big|_{\mc T}$.
\end{lemma}
\begin{proof}
The first statement is immediate. For the second one, as $E(t) \in \mc M$, by Cauchy-Schwarz inequality we have
\begin{equation}
\label{eq:r3}
\big| \lan E(t), y(t)\,x(t)^* \big|_{\mc T} \ran \big| \le 1\;,
\end{equation}
where equality occurs if and only if $E(t) = y(t)\,x(t)^* \big|_{\mc T}$. Therefore, $E(t)=\bar E$ is an equilibrium of \eqref{eq:ode} if and only if $\bar E=y\,x^*\big|_{\mc T}$, which implies that $x$ and $y$ fulfil (\ref{p1}). 
\end{proof}
\begin{lemma}[Monotonicity property of the flow]
The solution of the differential equation {\rm (\ref{eq:ode})} 
is characterized by the following property for the 
rightmost eigenvalue $\lambda(t)$ of $A+\eps E(t)$:
\begin{eqnarray}
\Re \big( \dot{\lambda}(t) \big) & \ge & 0 \qquad \forall\, t \ge 0\;,
\label{eq:pos}
\nonumber
\end{eqnarray}
\label{th:posder}
where equality occurs if and only if $E(t)=\bar E$, with $\bar E$ an equilibrium.
\end{lemma}
\begin{proof}
Having assumed $y(t)^* x(t) > 0$ for $t \ge 0$, we have to show that 
\[
\Re\left( y(t)^* \dot{E}(t) x(t) \right) \ge 0\;.
\]
Observing that, by Lemma \ref{lem:p1},
\begin{eqnarray}
&& \Re \Big( y(t)^*\,\left( y(t)\,x(t)^* \big|_{\mc T} \right) \,x(t) \Big) = \big\| y(t) x(t)^* \big|_\mc S \big\|_F\;,
\label{eq:r1}
\nonumber
\\[0.2cm]
&& \big| y(t)^*\,E(t)\,x(t) \big| \le \big\| y(t) x(t)^* \big|_\mc S \big\|_F\;,
\label{eq:r2}
\nonumber
\end{eqnarray}
and using \eqref{eq:r3}, we get the result.
\end{proof}
\begin{theorem}
\label{th:2.5}
Assume that $\la$ is a local maximum on $\partial \Lambda_\eps(A)$; then $\lambda \in \Lambda(A + \eps E)$, where $E = y\,x^* \big|_{\mc T}$ with $x$ and $y$ satisfying \eqref{p1}.
\end{theorem}
\begin{proof}
We argue by contradiction and assume $\lambda \in \Lambda(A + \eps E)$ with $E \ne y\,x^* \big|_{\mc T}$. By Lemma \ref{lem:equil}, this implies that $E$ is not an equilibrium of \eqref{eq:ode}. Denoting by $E(t)$ the solution to \eqref{eq:ode} with initial datum 
$E_0=E$, by Lemma \ref{th:posder} we get the strict inequality $\dot{\lambda}(0)>0$, which would imply that $\la$ is not a local maximum. 
\end{proof}

\subsection{Pseudospectral radius}

We define 
\begin{eqnarray}
     \rhoeps^{\mc T} (A) & = & \max\{|\lambda|: \lambda \in \Lameps^{\mc T}(A) \}\;. 
     \label{eq:psr}
     \nonumber
\end{eqnarray}
the structured pseudospectral radius.
The following simple variant of Algorithm 1 allows to compute 
locally extremal points of the $\eps$-pseudospectrum, with maximal modulus.

\noindent{\bf Algorithm 2.} Let  $\lambda_0$ be an eigenvalue with largest modulus of a Toeplitz matrix $A\in \bb C^{n\times n}$ with corresponding right and left eigenvectors  $x_0$ and $y_0$ normalized to be RP-compatible. Set $B_1 = A+\eps\, {\rm e}^{\i \arg(\lambda_0)}y_0x_0^*|_\mc T$. 

For $k=1,2,\ldots$, let $\lambda_k$ be an eigenvalue with largest modulus of $B_k$ closest to $\lambda_{k-1}$. Let $x_k$ and $y_k$ be corresponding right and left eigenvectors normalized to be RP-compatible. Set $B_{k+1} = A+\eps\, {\rm e}^{\i \arg(\lambda_{k})}\, y_k x_k^*|_\mc T$. $\diamond$

\medskip

By the definition of Algorithm 2, it follows immediately that the fixed points of the associated map are given by the pairs $(x,y)$ solution to
\[
\left\{\begin{array}{l}
y^* \left(A+\eps\, {\rm e}^{\i \arg(\lambda)}\, yx^*|_\mc T \right)= \lambda y^*\;, \\[0.2cm] 
\left(A+\eps\, {\rm e}^{\i \arg(\lambda)}\, yx^*|_\mc T \right) x = \lambda x\;.
\end{array}\right.
\]
We next introduce the differential equation for $E(t)\in \mc M$ obtained by replacing $y(t)$ by $\lambda(t)  y(t)$ in the right-hand side of \eqref{eq:ode}. It is straightforward that the analog of Lemma \ref{th:posder} applies in the present context. Moreover, as
\[
\frac{d \big(|\lambda(t)|^2\big)}{d t} = 2  \Re \big( \bar\lambda(t) \dot\lambda(t) \big)\;,
\]
a monotonicity property for $|\lambda(t)|$ holds for such flow. Therefore, arguing analogously to the proof of Theorem \ref{th:2.5}, we can conclude that every point $\lambda \in \Lameps^{\mc T}(A)$ which locally maximizes $|\lambda|$, has to be a stationary point of Algorithm 2.

\subsection{Rotated computation}
\label{subs:rot}

In order to partially compute the boundary of the pseudospectrum, we can apply the algorithm to a rotated matrix ${\rm e}^{-i \theta} A$ to reach the boundary along the direction with angle ${}\theta$. Indeed we are able to compute rightmost points of the rotated pseudospectrum by our algorithm and draw them after a rotation back. This allows us to represent some convex sections of the boundary and draw a set which includes the pseudospectrum (see the subsequent Section \ref{sec:5} for some illustrative examples).

\subsection{Boundary of the $\boldsymbol{\eps}$-pseudospectrum} 

We are interested to investigate whe\-ther the number of rightmost points of the $\eps$-pseudospecrum has to be finite. This is done rigorously in the next section for the case of tridiagonal Toepliz matrices. Concerning the general case, we are able to show that such a number is finite at least when $\eps$ is small enough. Indeed, Theorem \ref{th:epssmall} shows that if $\eps$ is sufficiently small then the algorithm locally converges to its fixed points, whence they have to be isolated. 

\section{The tridiagonal case}
\label{sec:3}

In this section we study the system \eqref{p1} in the simpler case of tridiagonal Toeplitz matrices. We shall use the notation $T(s,d,t)$ for the tridiagonal Toeplitz matrix with $s,d,t$ as sub-diagonal, diagonal, and super-diagonal entries respectively. Recall that the spectrum of $T=T(s,d,t)$ is given by 
\begin{eqnarray}
\label{betterR}
\nonumber
\lambda _{h}(T) & = & d +2\sqrt{|st|}\rme^{\rmi (\arg s + \arg t)/2}\,\cos \frac{h\pi}{n+1}\,,\quad h=1:n\;.
\end{eqnarray}
Let $A=T(\sigma_0,\delta_0,\tau_0)  \in \bb C^{n\times n}$ with 
\begin{equation}
\label{ass}
\sigma_0\tau_0\ne 0,\qquad \frac{\arg\sigma_0 + \arg \tau_0}2 \ne \pm \frac\pi 2\;.
\end{equation}
We remark that the assumptions \eqref{ass} guarantee $A$ has $n$ simple eigenvalues lying on a not vertical segment. If a pair $(x,y)$ is solution to \eqref{p1} then they are the right and left eigenvectors of a rightmost eigenvalue of a tridiagonal Toeplitz matrix, say $T(\sigma,\delta,\tau)$. Therefore, by setting
\begin{eqnarray}
\label{ts}
\nonumber
\sigma & = & |\sigma|\rme^{i\alpha}, \quad \tau = |\tau|\rme^{i\beta}\;,
\end{eqnarray}
the vectors $x,y$ have components,
\begin{equation}
\label{p3}
\left\{\begin{array}{l}
{\displaystyle x_k = \rme^{\rmi (\alpha-\beta)/2} \left(\left|\frac{\sigma}{\tau}\right|\right)^{k/2} \sin\frac{k\pi r}{n+1}\;,} \\ \\
{\displaystyle y_k = \rme^{\rmi (\alpha-\beta)/2} \left(\left|\frac{\tau}{\sigma}\right|\right)^{k/2} \sin\frac{k\pi r}{n+1}\;,} \end{array}\right. \quad k=1:n\;,
\end{equation}
where, by \eqref{ass}, $r=1$ or $r=n$ depending on which one between the extremal eigenvalues of $A$ has the largest real part; indeed we tacitly assume the parameter $\eps$ to be small enough to not affect this property. Therefore,
\begin{equation}
\label{p4}
(yx^*)_{k,h} = \rme^{\rmi (k-h)(\alpha-\beta)/2}\left(\left|\frac{\tau}{\sigma}\right|\right)^{(k-h)/2}\sin\frac{k\pi r}{n+1} \sin\frac{h\pi r}{n+1}\;, \quad k,h=1:n\;.
\end{equation}
We notice that  $yx^*\big|_{\mc S} = T(\sigma_1,\delta_1,\tau_1)$ with
\begin{equation*}
\label{p6}
\begin{split}
\sigma_1 & = \rme^{\rmi (\alpha-\beta)/2}\sqrt{\left|\frac{\tau}{\sigma}\right|}\,\frac{n+1}{2(n-1)}\cos\frac{\pi r}{n+1}\; , \qquad \delta_1 = \frac{n+1}{2n}\;, \\
\tau_1 & = \rme^{\rmi (\beta-\alpha)/2}\sqrt{\left|\frac{\sigma}{\tau}\right|}\,\frac{n+1}{2(n-1)}\cos\frac{\pi r}{n+1}\;. 
\end{split}
\end{equation*}
Indeed, the arithmetic means of the diagonal terms in \eqref{p4} can be explicitly computed. More precisely,  
\begin{equation*}
\frac 1n \sum_{k=1}^{n-1} \sin^2\frac{k\pi r}{n+1} = \frac 1{2n} \Re\left\{\sum_{k=1}^{n-1} \left[1-\left(\rme^{\rmi\frac{2\pi r}{n+1}}\right)^k\right] \right\}= \frac{n+1}{2n}\;,
\end{equation*}
\begin{equation*}
\begin{split}
\frac{1}{n-1} & \sum_{k=1}^{n-1} \sin\frac{k\pi r}{n+1} \sin\frac{(k+1)\pi r}{n+1}  \\ & = \frac{1}{2(n-1)} \sum_{k=1}^{n-1} \left(\cos\frac{\pi r}{n+1} - \cos\frac{(2k+1)\pi r}{n+1}\right)\\ & = \frac 12 \cos\frac{\pi r}{n+1} - \frac{1}{2(n-1)} \Re \left\{ \rme^{\rmi\frac{\pi r}{n+1}} \sum_{k=1}^{n-1} \left(\rme^{\rmi\frac{2\pi r}{n+1}}\right)^k\right\} \\ & = \frac 12 \cos\frac{\pi r}{n+1} + \frac{1}{2(n-1)} \left(\sin\frac{\pi r}{n+1}\right)^{-1}\sin\frac{2\pi r}{n+1} \\ & = \frac{n+1}{2(n-1)}\cos\frac{\pi r}{n+1}\;.
\end{split}
\end{equation*}
Moreover,
\begin{equation*}
\label{p7}
\begin{split}
\|yx^*\big|_{\mc S}\|_F  & = \sqrt{n|\delta_1|^2 + (n-1)\left(|\sigma_1|^2+|\tau_1|^2\right)} \\ & = \frac{n+1}2 \sqrt{ \frac 1n + \frac{1}{n-1}  \left(\left|\frac{\sigma}{\tau}\right| + \left|\frac{\tau}{\sigma}\right|\right)\cos^2\frac{\pi r}{n+1}}\;.
\end{split}
\end{equation*}
In conclusion,
\begin{equation*}
\label{p8}
A+\eps yx^*\big|_{\mc T} = T(\sigma_0+\eps\hat\sigma,\delta_0+\eps\hat\delta,\tau_0 +\eps\hat\tau)\:,
\end{equation*}
where
\begin{equation}
\label{p9}
\hat\sigma = \frac{\rme^{\rmi (\alpha-\beta)/2}}{n-1} \sqrt{\left|\frac{\tau}{\sigma}\right|}\cos\frac{\pi r}{n+1} \left[\frac 1n + \frac{1}{n-1}  \left(\left|\frac{\sigma}{\tau}\right| + \left|\frac{\tau}{\sigma}\right|\right)\cos^2\frac{\pi r}{n+1}\right]^{-1/2},
\end{equation}
\begin{equation*}
\label{p10}
\hat\delta = \frac 1n \left[\frac 1n + \frac{1}{n-1}  \left(\left|\frac{\sigma}{\tau}\right| + \left|\frac{\tau}{\sigma}\right|\right)\cos^2\frac{\pi r}{n+1}\right]^{-1/2},
\end{equation*}
\begin{equation}
\label{p11}
\hat\tau = \frac{\rme^{\rmi (\beta-\alpha)/2}}{n-1} \sqrt{\left|\frac{\sigma}{\tau}\right|}\cos\frac{\pi r}{n+1}\left[\frac 1n + \frac{1}{n-1}  \left(\left|\frac{\sigma}{\tau}\right| + \left|\frac{\tau}{\sigma}\right|\right)\cos^2\frac{\pi r}{n+1}\right]^{-1/2}.
\end{equation}
By the characterization \eqref{p3} of the eigenvectors of a tridiagonal Toeplitz matrix, for $(x,y)$ to be solution to \eqref{p1}, the parameters $\sigma$ and $\tau$ have to satisfy the following relations,
\begin{equation}
\label{p12}
\left\{\begin{array}{l} {\displaystyle \sqrt{\left|\frac{\sigma_0+\eps\hat\sigma}{\tau_0+\eps\hat\tau}\right|} = \sqrt{\left|\frac{\sigma}{\tau}\right|}\;,}  \\[0.5cm] {\displaystyle \exp\left(\rmi \frac{\arg(\sigma_0+\eps\hat\sigma) - \arg (\tau_0+\eps\hat\tau)}2\right) = \exp\left(\rmi \frac{\alpha-\beta}2\right)\;.} \end{array}\right.
\end{equation}
The system \eqref{p12} can by analyzed by considering the following complex equation,
\begin{equation}
\label{p121}
\frac{\sigma_0+\eps\hat\sigma}{\tau_0+\eps\hat\tau} = \left|\frac{\sigma}{\tau}\right|\rme^{\rmi(\alpha-\beta)}\;,
\end{equation}
whose solutions solve either system \eqref{p12} or
\begin{equation*}
\left\{\begin{array}{l} {\displaystyle \sqrt{\left|\frac{\sigma_0+\eps\hat\sigma}{\tau_0+\eps\hat\tau}\right|} = \sqrt{\left|\frac{\sigma}{\tau}\right|}\;,}  \\[0.5cm] {\displaystyle \exp\left(\rmi \frac{\arg(\sigma_0+\eps\hat\sigma) - \arg (\tau_0+\eps\hat\tau)}2\right) = - \exp\left(\rmi \frac{\alpha-\beta}2\right)\;.} \end{array}\right.
\end{equation*}
Therefore, it suffices  to solve \eqref{p121} with the costraint
\begin{equation}
\label{p121b}
\exp\left(\rmi \frac{\arg(\sigma_0+\eps\hat\sigma) - \arg (\tau_0+\eps\hat\tau)}2\right) = \exp\left(\rmi \frac{\alpha-\beta}2\right)\;.
\end{equation}
Setting
\begin{equation}
\label{p122}
\varrho = \left|\frac{\sigma}{\tau}\right|, \quad \varphi = \frac{\alpha-\beta}2, \quad a = \frac{\eps\sqrt n}{n-1}\cos\frac{\pi r}{n+1}, \quad b = \frac{n}{n-1}\cos^2\frac{\pi r}{n+1}\;,
\end{equation}
by \eqref{p9} and \eqref{p11} we have
\begin{equation}
\label{hsht}
\eps\hat\sigma = \frac{a\rme^{\rmi \varphi}}{\sqrt{\varrho+b(1+\varrho^2)}}, \qquad \eps\hat\tau = \frac{a\varrho\rme^{-\rmi \varphi}}{\sqrt{\varrho+b(1+\varrho^2)}}\;.
\end{equation}
Substituting in \eqref{p121}, after some easy computations the latter reads,
\begin{equation}
\label{p123}
\varrho\tau_0\rme^{2\rmi\varphi}+\frac{a(\varrho^2-1)}{\sqrt{\varrho+b(1+\varrho^2)}}\rme^{\rmi \varphi} - \sigma_0 = 0\;. 
\end{equation}
Hence, defining
\begin{equation}
\label{p124}
G(\varrho) = \frac{a(1-\varrho^2)}{2\varrho\sqrt{\varrho+b(1+\varrho^2)}}\;,
\end{equation}
equation \eqref{p123} becomes
\begin{equation}
\label{p125}
\rme^{\rmi\varphi} = \frac{G(\varrho)}{\tau_0} \pm \frac{1}{\tau_0} \sqrt{G(\varrho)^2 +\frac{\sigma_0\tau_0}\varrho}\;. 
\end{equation}
\begin{theorem}
\label{teo:1}
For each $\sigma_0,\tau_0\in\bb C$ satisfying \eqref{ass} and for each $n>1$ there exists $\eps_n=\eps_n(\sigma_0,\tau_0)$ such that  for any  $\eps\in [0,\eps_n)$ there is a unique pair $(\varrho_*,\varphi_*)$ solution to \eqref{p125} which verify \eqref{p121b} with $(\alpha-\beta)/2 = \varphi_*$. Moreover $\eps_n \to \infty$ per $n\to\infty$.
\end{theorem} 
Since
\begin{equation}
\label{ab}
|a| \le \frac{2\eps}{\sqrt n}, \quad \frac 12 \le b \le \frac 32 \qquad \forall\, n>1\;,
\end{equation}
we can analyze \eqref{p125} under the assumptions that $b \in \left[\frac 12 ,\frac 32 \right]$ and $|a|$ is small enough. Theorem \ref{teo:1} results to be an easy corollary of the proposition below.
\begin{proposition}
\label{prop:1}
Let
\begin{equation}
\label{p126}
F_\pm(\varrho) = \frac{G(\varrho)}{\tau_0} \pm \frac{1}{\tau_0} \sqrt{G(\varrho)^2 +\frac{\sigma_0\tau_0}\varrho}\;.
\end{equation}
For each $\sigma_0,\tau_0\in\bb C$ there exists $a_0>0$ such that for any $|a|\in(0,a_0)$ and  $b \in \left[\frac 12 ,\frac 32 \right]$ there is a unique positive $\varrho_+$ [resp.\ $\varrho_-$] such that $|F_+(\varrho_+)|=1$ [resp.\ $|F_-(\varrho_-)|=1$]. Moreover,

\smallskip
1) If $\Re(\sigma_0\tau_0) > - |\sigma_0| |\tau_0|$ then
\begin{equation}
\label{1ca}
r=1 \quad \Longrightarrow \quad \begin{cases}
\varrho_-<\varrho_+<1  & \text{if}\quad|\sigma_0|<|\tau_0|\;,\ \\[0.3cm] 1<\varrho_+<\varrho_-  & \text{if}\quad |\sigma_0|>|\tau_0|\;.
\end{cases}
\end{equation}
\begin{equation}
\label{2ca}
r=n \quad \Longrightarrow \quad \begin{cases}
\varrho_+<\varrho_-<1  & \text{if}\quad|\sigma_0|<|\tau_0|\;, \\[0.3cm] 1<\varrho_-<\varrho_+  & \text{if}\quad |\sigma_0|>|\tau_0|\;.
\end{cases}
\end{equation}
Finally, if $|\sigma_0|=|\tau_0|=1$ then $\varrho_+=\varrho_-=1$.

\smallskip
2) If $\Re(\sigma_0\tau_0) = - |\sigma_0| |\tau_0|$ then $\varrho_+=\varrho_- = \displaystyle{\frac{|\sigma_0|}{|\tau_0|}} $.
\end{proposition}
\begin{proof} 
The square root appearing in \eqref{p126} is intended to be the principal one. Otherwise stated, if $z = r\rme^{\rmi \theta}$ with $\theta\in (-\pi, \pi]$ then $\sqrt z = \sqrt r \rme^{\rmi \theta/2}$. In particular, this implies 
\begin{equation*}
\overline{\sqrt z} = \begin{cases}
\sqrt{\bar z} & \text{if $\Re (z) \ne -|z|$}\;, \\[0.3cm]
- \rmi \sqrt{|z|} & \text{if $\Re (z)= -|z|$}\;.
\end{cases}
\end{equation*}
We study separately the following two cases.

\medskip
\noindent Case 1): $\Re(\sigma_0\tau_0) \ne - |\sigma_0| |\tau_0|$. We have,
\begin{equation*}
|F_\pm(\varrho)|^2 = \frac{1}{|\tau_0|^2} \left(G(\varrho)\pm \sqrt{G(\varrho)^2 +\frac{\sigma_0\tau_0}\varrho}\right) \left(G(\varrho)\pm \sqrt{G(\varrho)^2 +\frac{\bar\sigma_0\bar\tau_0}\varrho}\right)\;, 
\end{equation*}
whence the equation $|F_\pm(\varrho)|^2=1$  reads,
\begin{equation*}
 G(\varrho)\pm \sqrt{G(\varrho)^2 +\frac{\bar\sigma_0\bar\tau_0}\varrho} =
|\tau_0|^2\left(G(\varrho)\pm \sqrt{G(\varrho)^2 +\frac{\sigma_0\tau_0}\varrho}\right)^{-1}\;,
\end{equation*}
which can be recasted into the form,
\begin{equation*}
 G(\varrho)\pm \sqrt{G(\varrho)^2 +\frac{\bar\sigma_0\bar\tau_0}\varrho} =
-\frac{\varrho\bar\tau_0}{\sigma_0}\left(G(\varrho)\mp \sqrt{G(\varrho)^2 +\frac{\sigma_0\tau_0}\varrho}\right)\;,
\end{equation*}
that is
\begin{equation}
\label{cc}
\left(1+\frac{\varrho\bar\tau_0}{\sigma_0}\right)G(\varrho) = 
\pm \frac{\varrho\bar\tau_0}{\sigma_0} \sqrt{G(\varrho)^2 +\frac{\sigma_0\tau_0}\varrho} \mp \sqrt{G(\varrho)^2 +\frac{\bar\sigma_0\bar\tau_0}\varrho}\;.
\end{equation}
Rationalizing, after some computations we obtain,
\begin{equation*}
4\varrho^2G(\varrho)^2 = \frac{(|\tau_0|^2\varrho^2-|\sigma_0|^2)^2}{|\tau_0|^2\varrho^2+2 \Re(\sigma_0\tau_0)\varrho + |\sigma_0|^2}\;.
\end{equation*}
Plugging the definition \eqref{p124} of $G(\varrho)$ we finally get,
\begin{equation}
\label{p127}
\frac{a^2(1-\varrho^2)^2}{b\varrho^2 + \varrho+b} = \frac{(|\tau_0|^2\varrho^2-|\sigma_0|^2)^2}{|\tau_0|^2\varrho^2+2 \Re(\sigma_0\tau_0)\varrho + |\sigma_0|^2}\;.
\end{equation}
Let us consider the functions appearing in \eqref{p127}, that is
\begin{equation*}
f_1(\varrho) = \frac{a^2(1-\varrho^2)^2}{b\varrho^2 + \varrho+b}\; , \qquad
f_2(\varrho) = \frac{(|\tau_0|^2\varrho^2-|\sigma_0|^2)^2}{|\tau_0|^2\varrho^2+2 \Re(\sigma_0\tau_0)\varrho + |\sigma_0|^2}\;,
\end{equation*}
restricted on the domain of interest $\{\varrho>0\}$. We claim that if $|\sigma_0| \ne |\tau_0|$ then the corresponding graphs intersect each other in two points whose abscissae  $\varrho_1, \varrho_2$ are such that
\begin{equation}
\label{pp12}
\begin{cases}
{\displaystyle\varrho_1<\frac{|\sigma_0|}{|\tau_0|}<\varrho_2<1}  & \text{if}\quad |\sigma_0|<|\tau_0|\;, \\[0.5cm] {\displaystyle 1<\varrho_1<\frac{|\sigma_0|}{|\tau_0|}<\varrho_2} & \text{if}\quad |\sigma_0|>|\tau_0|\;,
\end{cases}
\end{equation}
while if $|\sigma_0| = |\tau_0|$ such graphs intersect each other solely in the point $(1,0)$.  

We start by noticing that $f_1(\varrho)$ reaches his minimum value uniquely in $\varrho=1$. More precisely, it decreases in  $[0,1]$, from  $f_1(0) = a^2/b$ to $f_1(1)=0$ and increases in $[1,+\infty)$, diverging with $\varrho^{-2} f_1(\varrho) \to a^2/b$ as $\varrho\to\infty$.
\par\noindent
Concerning $f_2(\varrho)$, if $|a|$ is small enough then
\begin{equation}
\label{f12}
f_2(0) = |\sigma_0|^2 > f_1(0)\;, \qquad \lim_{\varrho\to\infty} \varrho^{-2} f_2(\varrho) = |\tau_0|^2 > \lim_{\varrho\to\infty} \varrho^{-2} f_1(\varrho)\;.
\end{equation}
We now distinguish the cases $\Im(\sigma_0\tau_0)=0$ and $\Im(\sigma_0\tau_0)\ne 0$. In the first case $\Re(\sigma_0\tau_0) = |\sigma_0||\tau_0|$, whence
\begin{equation*}
f_2(\varrho) = (|\tau_0|\varrho -|\sigma_0|)^2\;, 
\end{equation*}
which is the law of a parabola with vertex in $(\frac{|\sigma_0|}{|\tau_0|},0)$. Therefore, by \eqref{f12}, the claim is straightforward. In the second case we have,
\begin{equation*}
f_2(\varrho) = \frac{(|\tau_0|^2\varrho^2-|\sigma_0|^2)^2}{(|\tau_0|\varrho +\Re(\sigma_0\tau_0)/|\tau_0|)^2 + \kappa^2}\;,
\end{equation*}
with $\kappa^2=|\sigma_0|^2- \Re(\sigma_0\tau_0)^2/|\tau_0|^2  >0$. As a function on the whole line, $f_2$ has two absolute minima for $\varrho = \pm \frac{|\sigma_0|}{|\tau_0|}$. In the interval $|\varrho| < \frac{|\sigma_0|}{|\tau_0|}$ there can be one or three local extrema. But in any cases, by choosing $|a|$ small enough, the claim is easily verified.

We are left with showing that $\varrho_1$ and $\varrho_2$ solve $|F_+(\varrho_1)|=|F_-(\varrho_2)|=1$ or $|F_+(\varrho_2)|=|F_-(\varrho_1)|=1$, thus proving the proposition (in the case $|\sigma_0|=|\tau_0|$ the identity $F_\pm(1)=1$ is immediate). Taking the real part of \eqref{cc} with $\varrho=\varrho_1$ we have,
\begin{equation}
\label{pp}
(1+\gamma_1\cos\theta)G(\varrho_1) = \pm \big(\gamma_1\cos\theta\,\Re(H_1) - \gamma_1\sin\theta\,\Im(H_1) - \Re (H_1)\big)\;,
\end{equation}
where
\begin{equation*}
\gamma_1 = \varrho_1 \frac{|\tau_0|}{|\sigma_0|}\;, \quad \theta = \arg\sigma_0+\arg\tau_0\;, \quad H_1 = \sqrt{G(\varrho_1)^2 +\frac{\sigma_0\tau_0}{\varrho_1}}\;.
\end{equation*}
By \eqref{p124}, \eqref{pp12}, and recalling that by the definiton of $a$, see \eqref{p122}, we have $a>0$ if $r=1$ and $a<0$ if $r=n$, we conclude that
\begin{equation*}
G(\varrho_1)\begin{cases}
>0 & \text{if $r=1$ and $|\sigma_0|<|\tau_0|$ or $r=n$ and $|\sigma_0|>|\tau_0|$}\;, \\[0.3cm]
<0 & \text{if $r=1$ and $|\sigma_0|>|\tau_0|$ or $r=n$ and $|\sigma_0|<|\tau_0|$}\;.
\end{cases}
\end{equation*}
Moreover, as $\gamma_1<1$, the left-hand side of \eqref{pp} has the same sign as $G(\varrho_1)$. On the other hand, since $\Im (H_1^2)$ has the same sign as $\sin\theta$,  we have $\Re (H_1)>0$ and $\sin\theta\, \Im (H_1)\ge 0$, so that 
\begin{equation*}
\gamma_1\cos\theta\,\Re (H_1) - \gamma_1\sin\theta\,\Im (H_1) - \Re (H_1) \le (\gamma_1\cos\theta-1)\,\Re (H_1) <0 \;.
\end{equation*}
By \eqref{pp} it follows that

\smallskip
i) $\varrho_1$ cannot be solution of  $|F_+(\varrho)|=1$ if $r=1$ and $|\sigma_0|<|\tau_0|$ or $r=n$ and $|\sigma_0|>|\tau_0|$; therefore  $\varrho_-=\varrho_1$ and $\varrho_+=\varrho_2$ in these cases.

\smallskip
ii) $\varrho_1$ cannot be solution of  $|F_-(\varrho)|=1$ if $r=1$ and $|\sigma_0|>|\tau_0|$ or $r=n$ and $|\sigma_0|<|\tau_0|$;  therefore $\varrho_-=\varrho_2$ and $\varrho_+=\varrho_1$ in these cases.

\medskip
\noindent Case 2): $\Re(\sigma_0\tau_0) = - |\sigma_0| |\tau_0|$. We have, 
\begin{equation*}
F_\pm(\varrho) = \frac{1}{\tau_0} \left(G(\varrho)\pm \sqrt{G(\varrho)^2 -\frac{|\sigma_0| |\tau_0|}\varrho}\right)\;,
\end{equation*}
where
\begin{equation*}
\sqrt{G(\varrho)^2 -\frac{|\sigma_0| |\tau_0|}\varrho} = \begin{cases}
{\displaystyle \rmi \, \sqrt{\frac{|\sigma_0| |\tau_0|}\varrho-G(\varrho)^2}}    & \text{if} \;\; {\displaystyle G(\varrho)^2 \le \frac{|\sigma_0| |\tau_0|}\varrho}\;,  \\ \\
{\displaystyle \sqrt{\bigg|G(\varrho)^2 -\frac{|\sigma_0| |\tau_0|}\varrho}\bigg|}  & \text{if}\;\; {\displaystyle G(\varrho)^2 > \frac{|\sigma_0| |\tau_0|}\varrho}\;. 
\end{cases}
\end{equation*}
Therefore 
\begin{equation}
\label{p111}
|F_\pm(\varrho)|^2 = \frac{|\sigma_0|}{|\tau_0|\varrho} \, \mathcal{X}_{\{G(\varrho)^2 \varrho \le |\sigma_0| |\tau_0|\}} +  K_\pm(\varrho)\; \mathcal{X}_{\{G(\varrho)^2 \varrho > |\sigma_0| |\tau_0|\}}\;,
\end{equation}
where $\mathcal{X}_{\{\cdot\}}$ denotes the characteristic set function, and  
\begin{equation*}
K_\pm(\varrho) = \frac{1}{|\tau_0|^2} \left(2G(\varrho)^2 - \frac{|\sigma_0| |\tau_0|}\varrho  \pm 2 G(\varrho)\sqrt{G(\varrho)^2 -\frac{|\sigma_0| |\tau_0|}\varrho} \right)\;.
\end{equation*}
We now observe that the equation $K_\pm(\varrho) =1$ can be written in the form,
\begin{equation*}
|\tau_0|^2 - 2G(\varrho)^2 + \frac{|\sigma_0| |\tau_0|}\varrho = \pm 2 G(\varrho)\sqrt{G(\varrho)^2 -\frac{|\sigma_0| |\tau_0|}\varrho}\;,
\end{equation*}
from which, recalling the definition of $f_1(\varrho)$, we get
\begin{equation*}
f_1(\varrho) = (|\tau_0|\varrho+|\sigma_0|)^2\;.
\end{equation*}
By the previous qualitative analysis of $f_1(\varrho)$ we easily deduce that such equation does not have positive solutions for  $|a|$ small. On the other hand,  if $|a|$ is sufficiently small then the condition $G(\varrho)^2 \varrho \le |\sigma_0| |\tau_0|$ is fulfilled by $\varrho = \frac{|\sigma_0|}{|\tau_0|}$, whence by \eqref{p111} we get the result.
\end{proof}
\begin{remark}
\label{rem:p1}
It is worthwile to notice that even the case $\Re(\sigma_0\tau_0) = |\sigma_0||\tau_0|$ is quite explicit. Indeed, since $a>0$ and $4\varrho^2G(\varrho)^2 = (|\tau_0|\varrho -|\sigma_0|)^2$, 
\begin{equation*}
G(\varrho_i) = \begin{cases}
{\displaystyle \frac{|\sigma_0|-|\tau_0|\varrho_i}{2\varrho_i}} & \text{if $i=1$ and $|\sigma_0|<|\tau_0|$ or $i=2$ and $|\sigma_0|>|\tau_0|$}\;, \\[0.3cm] {\displaystyle \frac{|\tau_0|\varrho_i-|\sigma_0|}{2\varrho_i}} & \text{if $i=1$ and $|\sigma_0|>|\tau_0|$ or $i=2$ and $|\sigma_0|<|\tau_0|$}\;.
\end{cases}
\end{equation*}
Plugging these values in \eqref{p126}, since $\sigma_0\tau_0 = |\sigma_0| |\tau_0|$, in the first case we easily get
\begin{equation}
\label{fr1}
F_\pm(\varrho_i) = \frac{|\sigma_0|-|\tau_0|\varrho_i}{2\tau_0\varrho_i} \pm \frac{|\sigma_0|+|\tau_0|\varrho_i}{2\tau_0\varrho_i} = \begin{cases}
{\displaystyle \frac{|\sigma_0|}{\tau_0\varrho_i}} & \text{if $+$}\;, \\ \\
{\displaystyle - \frac{|\tau_0|}{\tau_0}}   & \text{if $-$}\;,
\end{cases}
\end{equation}
while in the second case,
\begin{equation}
\label{fr2}
F_\pm(\varrho_i) = \frac{|\tau_0|\varrho_i-|\sigma_0|}{2\tau_0\varrho_i} \pm \frac{|\tau_0|\varrho_i+|\sigma_0|}{2\tau_0\varrho_i} = \begin{cases}
{\displaystyle \frac{|\tau_0|}{\tau_0}} & \text{if $+$}\;, \\ \\
{\displaystyle - \frac{|\sigma_0|}{\tau_0\varrho_i}} & \text{if $-$}\;.
\end{cases}
\end{equation}
\end{remark}
\medskip
\noindent
{\it Proof of Theorem~\ref{teo:1}}. We show that, setting  $\varphi_\pm=\arg F_\pm(\varrho_\pm)$,  for any $|a|$ small enough we have, 
\begin{equation}
\label{p121b*}
\exp\left(\rmi \frac{\arg(\sigma_0+\eps\hat\sigma_\pm) - \arg (\tau_0+\eps\hat\tau_\pm)}2\right) = \begin{cases}
\pm \exp\left(\rmi \varphi_\pm\right) & \text{if $r=1$}\;, \\
\mp \exp\left(\rmi \varphi_\pm\right)   & \text{if $r=n$}\;,
\end{cases}
\end{equation}
where $\hat\sigma_\pm, \hat\tau_\pm$ are defined by $\hat\sigma,\hat\tau$ as in \eqref{hsht} and evaluated for $(\varrho,\varphi)=(\varrho_\pm,\varphi_\pm)$. By \eqref{p121b*} the statement of the theorem follows with $(\varrho_*,\varphi_*) = (\varrho_+,\varphi_+)$ if $r=1$ and $(\varrho_*,\varphi_*) = (\varrho_-,\varphi_-)$ if $r=n$. Moreover, since $|a|\le 2\eps/\sqrt n$, this also shows that the threshold $\eps_n$ can be chosen arbitrarily large increasing the dimension $n$.

To prove \eqref{p121b*} it suffices to observe that 
\begin{equation*}
\varrho_\pm = \frac{|\sigma_0|}{|\tau_0|} + o(1), \quad G(\varrho_\pm) = o(1)\;, \quad \eps\hat\sigma_\pm = o(1)\;, \quad \eps\hat\tau_\pm =o(1)\;,
\end{equation*}
where $o(1)$ stands for a generic function vanishing as $|a|\to 0$. Therefore, by \eqref{p126} it follows 
\begin{equation*}
F_\pm(\varrho_\pm) = \pm \frac{1}{\tau_0} \sqrt{\frac{|\tau_0|}{|\sigma_0|}\sigma_0\tau_0}+o(1) = \pm \exp\left(\rmi\frac{\arg(\sigma_0\tau_0) - 2\arg\tau_0}2\right) + o(1)\;.
\end{equation*}
On the other hand,
\begin{equation*}
\exp\left(\rmi\frac{\arg(\sigma_0+\eps\hat\sigma_\pm) - \arg (\tau_0+\eps\hat\tau_\pm)}2\right) = \exp\left(\rmi\frac{\arg\sigma_0-\arg\tau_0}2\right) + o(1)\;.
\end{equation*}
The equation \eqref{p121b*} now follows by noticing that  $\arg(\sigma_0\tau_0) = \arg\sigma_0+\arg\tau_0$ if $r=1$ while $\arg(\sigma_0\tau_0) = \arg\sigma_0+\arg\tau_0 \pm 2\pi$ if $r=n$.
\qed
 
\begin{remark}
\label{rem:p2} 
In case 2) of Proposition~\ref{prop:1}, i.e.\ $\sigma_0\tau_0 = - |\sigma_0| |\tau_0|$, the second condition in \eqref{ass} is not satisfied and the choice of $r$ cannot be established apriori. However, we observe that in such case $\varrho_+=\varrho_- = \frac{|\sigma_0|}{|\tau_0|}$,
\begin{equation*}
F_\pm\bigg(\frac{|\sigma_0|}{|\tau_0|}\bigg) = \frac{\rme^{-\rmi\arg\tau_0}}{|\tau_0|} \left(G\bigg(\frac{|\sigma_0|}{|\tau_0|}\bigg) \pm \rmi \sqrt{|\tau_0|^2 - G\bigg(\frac{|\sigma_0|}{|\tau_0|}\bigg)^2}\right)\;,
\end{equation*}
and $\arg\sigma_0 - \arg \tau_0 - \pi = -2\arg\tau_0$. Therefore,
\begin{equation*}
F_\pm\bigg(\frac{|\sigma_0|}{|\tau_0|}\bigg) = \frac{\rme^{\rmi(\arg\sigma_0-\arg\tau_0)/2}}{|\tau_0|} \left(\pm  \sqrt{|\tau_0|^2 - G\bigg(\frac{|\sigma_0|}{|\tau_0|}\bigg)^2}+ \rmi G\bigg(\frac{|\sigma_0|}{|\tau_0|}\bigg) \right)\;.
\end{equation*}
We conclude that Theorem~\ref{teo:1} holds also in this case, and precisely with $(\varrho_*,\varphi_*) = \Big(\frac{|\sigma_0|}{|\tau_0|}, \arg F_+\Big(\frac{|\sigma_0|}{|\tau_0|}\Big)\Big)$. 
\end{remark}

\section{Local Error Analysis}
\label{sec:4}

The aim of this section is to provide for a Toeplitz matrix $A$ with an arbitrary banded structure, a local error analysis of Algorithm 1 close to a simple locally rightmost point. 

The analysis presents similarities but also some additional difficulties with respect to that given
in \cite{GO11} for the unstructured case. In order to proceed we recall the definition of group inverse which we need in the analysis. 

\begin{definition}
\label{def:groupinv}
The \emph{group inverse} of a matrix $C$, denoted $C^\#$, is the unique matrix $G$ satisfying $CG =GC$,
$GCG=G$ and $CGC=C$.
\end{definition}
The following result is important for the error analysis of Algorithm 1. 
\begin{theorem} 
\label{th:localerr}
Suppose that $(x,y)$ is a boundary fixed point of the map $M_\eps$ corresponding to a simple rightmost eigenvalue $\lambda$ of $B=A+\epsyxstar \big|_\mc T$. Let the sequence $B_k$ and $L_k=y_k x_k^* \big|_\mc T$ be defined as in the map $M_\eps$ and $L=y x^* \big|_\mc T$ be the fixed point. Set
\begin{eqnarray}
F_k & = & \eps \left(y_{k-1} x_{k-1}^* - y x^*\right)
\nonumber
\\[0.2cm]
\label{eq:Edef}
\nonumber  
E_{k} & = & \eps  \bigl(y_{k-1} x_{k-1}^*\big|_\mc T - y x^*\big|_\mc T \bigr) = 
\eps  \left(L_{k-1} - L \right) \;,
\end{eqnarray}
for $k=1,2,...$ and let $\delta_k = \|E_k\|_2$.
Then we have
\begin{equation}
\label{eq:errorformula}
F_{k+1} = \eps\Big(\Re (\beta_k + \gamma_k) L - L  E_k^*G^* - G^* E_k^* L\Big) + \bigo(\delta_k^2)\;,
\end{equation}
where
\begin{equation*}
G=(B- \lambda I)^\#, \quad \beta_k=x^* G E_k x\;,\quad \gamma_k=y^*E_k G  y\;.
\label{eq:betagamma}
\end{equation*}
\end{theorem}
\begin{proof}
We omit the proof, which is similar to Guglielmi and Overton (see \cite[Theorem 5.3]{GO11})
and is based on the perturbative analysis of eigenvectors in \cite{MS88}. 
\end{proof}
Nevertheless there is an important difference with respect to the result given in \cite{GO11}.
Observe in fact that in the result given there we have $E_{k+1}$ replacing $F_{k+1}$
in the left-hand side of (\ref{eq:errorformula}) so that it is possible to study directly
the map $E_{k+1}(E_k)$. 

In the present case, however, since 
$
E_k \ne F_k\big|_\mc T\;,
$
the result (\ref{eq:errorformula}) has to be further elaborated. In particular, in order to proceed, we need the following lemma.

\begin{lemma}
Let $x,y$ and $\hat{x}, \hat{y}$ be RP-compatible and 
\[
\| y\,x^* - \hat{y}\,\hat{x}^* \|_2 \le \delta\;,
\] 
with $\delta$ sufficiently small. Then
\[
\| y\,x^*\big|_{\mc T} - \hat{y}\,\hat{x}^*\big|_\mc T \|_2\le \const\,\delta\;, 
\]
where $\const$ is a constant not depending on $\delta$.
\label{lem:proj}
\end{lemma}
\begin{proof}
We recall the following bounds, valid for any $M\in\bb C^{n\times n}$, 
\begin{equation}
\label{p}
\|M\|_2\le \|M\|_F \le \sqrt n \|M\|_2, \qquad \|M|_\mathcal{S}\|_F \le \|M\|_F\;,
\end{equation}
and observe that, by neglecting the off-diagonal terms contribution to the Frobenius norm, 
$$
\|yx^*|_\mathcal{S}\|_F\ge \frac{y^*x}{\sqrt n}\;.
$$
Therefore,
\begin{eqnarray*}
\|\hat y \hat x^*|_\mathcal{S}\|_F & \ge & \|yx^*|_\mathcal{S}\|_F - \|yx^*|_\mathcal{S}-\hat y \hat x^*|_\mathcal{S}\|_F \ge \frac{y^*x}{\sqrt n}  - \| yx^* - \hat y \hat x^*\|_F \\ & \ge & \frac{y^*x}{\sqrt n}   - \delta\sqrt n\;.
\end{eqnarray*}
Moreover,
\[
\begin{split}
\|yx^*|_\mathcal{T} - \hat y \hat x^*|_\mathcal{T}\|_2 & = \left\|\frac{yx^*|_\mathcal{S}}{\|yx^*\|_F} - \frac{\hat y \hat x^*|_\mathcal{S}}{\|\hat y \hat x^*\|_F} \right\|_2 \\ & = \left\|\frac{yx^*|_\mathcal{S} - \hat y \hat x^*|_\mathcal{S}}{\|yx^*\|_F} + \hat y \hat x^*|_\mathcal{S} \left(\frac 1{\|yx^*\|_F} - \frac 1{\|\hat y \hat x^*\|_F}\right)\right\|_2 \\ & \le   \frac{\|yx^*|_\mathcal{S} - \hat y \hat x^*|_\mathcal{S}\|_2}{\|yx^*|_\mathcal{S}\|_F} + \frac{\|\hat y \hat x^*|_\mathcal{S}\|_2}{\|\hat y \hat x^*|_\mathcal{S}\|_F} \frac{\big| \|\hat y \hat x^*|_\mathcal{S}\|_F - \| yx^*|_\mathcal{S}\|_F\big|}{\|yx^*|_\mathcal{S}\|_F} \\ & \le  \frac{\|yx^*|_\mathcal{S} - \hat y \hat x^*|_\mathcal{S}\|_2}{\|yx^*|_\mathcal{S}\|_F} + \frac{\|\hat y \hat x^*|_\mathcal{S}\|_2}{\|\hat y \hat x^*|_\mathcal{S}\|_F} \frac{\|yx^*|_\mathcal{S}- \hat y \hat x^*|_\mathcal{S}\|_F}{\|yx^*|_\mathcal{S}\|_F} \\ & \le \frac{\|yx^*|_\mathcal{S} - \hat y \hat x^*|_\mathcal{S}\|_F}{\|yx^*|_\mathcal{S}\|_F} \left(1+ \frac{\|\hat y \hat x^*|_\mathcal{S}\|_2}{\|\hat y \hat x^*|_\mathcal{S}\|_F} \right)\;.
\end{split}
\]
By \eqref{p} it follows that
$$
\frac{\|\hat y \hat x^*|_\mathcal{S}\|_2}{\|\hat y \hat x^*|_\mathcal{S}\|_F} \le 1\;, \qquad 
\|yx^*|_\mathcal{S}- \hat y \hat x^*|_\mathcal{S}\|_F \le \|yx^*- \hat y \hat x^* \|_F \le \sqrt n \|yx^*- \hat y \hat x^*\|_2\;.
$$
Therefore,
$$
\| y  x^*|_\mathcal{T} - \hat y \hat x^*|_\mathcal{T}\|_2 \le \frac{2\sqrt n}{\|yx^*|_\mathcal{S}\|_F}\,\|yx^*- \hat y \hat x^*\|_2 \le \frac{2\delta n}{y^*x}\;.
$$
\end{proof}

The following theorem establishes a useful formula for the group inverse of a singular matrix $C$.
\begin{theorem}[see \cite{GGO12}] 
\label{th:C}
Suppose that $C$ is singular and has a simple zero eigenvalue. Let 
the two vectors $x \in \ker(C)$ and $y \in \ker(C')$ be normalized 
so that $\|x\|_2 = \| y \|_2 = 1$. Let $C = U S V^*$, 
where $S = \diag(s_1,\ldots,s_{n-1},0)$, i.e.\ $s_i =\sigma_i(C)$, $i=1 : n-1$. 
Then
\begin{equation*}
G = C^\# = \left( I - w y^*\right) V \Xi U^*  \left( I - w y^*\right)\;,
\label{eq:groupinvformula}
\end{equation*}
where $w = \varrho x$ and $\displaystyle{\varrho={1}/{y^* x}}$,
so $y^* w=1$ and 
$$\Xi=
\diag(s_1^{-1},\ldots,s_{n-1}^{-1},0)\;.
$$ 

Moreover the following estimate holds,
\begin{eqnarray}
\| G \|_2 & \le & \frac{\varrho^2}{\sigma_{n-1}(C)}\;.
\label{eq:normG}
\end{eqnarray}
\end{theorem}

We can now establish a sufficient condition for local convergence.
\begin{theorem} 
\label{th:convergence}
Suppose that $(x,y)$ is a boundary fixed point of the map $M_\eps$ 
corresponding to a simple rightmost eigenvalue $\lambda$ of
$B=A+\epsyxstar\big|_\mc T$. Define
\begin{eqnarray}
    r=\frac{4\, \const\, \varrho^2\, \eps}{\sigma_{n-1}(A+\epsyxstar\big|_\mc T-\lambda I)}\;, \qquad \mbox{where} \ \varrho = \frac{1}{y^* x}
    \label{eq:convfac}
\end{eqnarray}
and $\const$ is the constant in Lemma {\rm \ref{lem:proj}}.
Then, if $r < 1$, and if $\delta_k=\|E_k\|_2$ is sufficiently small, then
$\lim_{j \rightarrow \infty} \lambda_{k+j} = \lambda$. Convergence is
at least linear with a rate less or equal to $r$.
\end{theorem}

\begin{proof}
Assume $\delta_k$ is sufficiently small. According to Theorem
\ref{th:localerr}, for studying local convergence we consider the 
map $\mc N_\eps$ defined by
\begin{equation*}
\mc F_{k+1} =  \mc N_\eps \left( \mc F_k \right) = \eps\Big(
\mathrm{Re} (\beta_k + \gamma_k) L - L \mc E_k^* G^* -
G^*\mc E_k^* L  \Big)\;, \nonumber
\end{equation*}
with $\beta_k=x^* G\mc E_k x$ and $\gamma_k=y^*\mc E_k G  y$, where $\mc E_k$
depends on $\mc F_k$. 
 
Since $\| L \|_2=1$, Lemma \ref{th:C} yields 
\[
\| \mc N_\eps \left( \mc F_k \right) \|_2 \le
\frac{4 \varrho^2 \eps}{\sigma_{n-1}(A+\epsyxstar\big|_\mc T-\lambda I)} \| \mc E_k
\|_2 \le r\| \mc F_k \|_2\;.
\]
The convergence of $\mc F_k$ clearly implies that of $\mc E_k$.

So, if $r<1$, the map $\mc N_\eps$ is a contraction, and
the sequence $\{ \lambda_{k} \}$ converges to $\lambda$ with a linear rate
bounded above by $r$.
\end{proof}
\begin{theorem} 
\label{th:epssmall}
Assume that $\lambda(0)$ is a simple rightmost eigenvalue of $A$
and $\lambda(\eps)$ is a path of boundary fixed points of the map $M_\eps$.
Then the bound $r(\eps)$ in {\rm (\ref{eq:convfac})} is such that 
$\lim\limits_{\eps \rightarrow 0} r(\eps)=0$. 
\end{theorem}
\begin{proof}
We put in evidence the dependence of the fixed point on $\eps$ and denote
by $x(\eps)$ and $y(\eps)$ the eigenvectors associated to the fixed point $\lambda(\eps)$.

First observe that 
$
\lim_{\eps \goes 0} \varrho(\eps) = {1}/{y(0)^*x(0)} > 0.
$
Furthermore we have
$$
\lim\limits_{\eps \goes 0} \sigma_{n-1}\left(B-\lambda(\eps)I\right) = \sigma_{n-1}\left(A-\lambda(0) I\right) > 0\;
$$  
by the simplicity assumption for the rightmost eigenvalue of $A$, which extends to $\lambda(\eps)$ for sufficiently
small $\eps$ by a continuity argument.
\end{proof}
Theorem \ref{th:epssmall} implies that for sufficiently small $\eps$, Algorithm 1 converges 
at least linearly with a rate $\mathcal{O}(\eps)$. Anyway, numerical experiments show that the method converges also for large values of $\eps$. 

\begin{remark}
\label{rem:p3} 
We notice that the parameter $\varrho$ appearing in Theorem \ref{th:epssmall} is the condition number of $\lambda$. In the case of tridiagonal Toeplitz matrices the condition number is computed in \cite[Eq.\ (20)]{NPR}. In particular, calling 
$$
m=\frac{\min\{|\sigma_0|,|\tau_0|\}}{\max\{|\sigma_0|,|\tau_0|\}}\;,
$$
one deduces that $\varrho\rightarrow 1$, as $m\rightarrow 1$, and $\varrho\rightarrow\infty$ as $m\rightarrow 0$. Moreover, in the latter case the following asymptotics holds, 
$$
\varrho\approx \frac{1-\cos \frac{2h\pi }{n+1}}{n+1}
\left( \frac{1}{m}\right)^{\frac{n-1}{2}},
$$
with $h=1$ or $h=n$ depending on the displacement of the spectrum of $A$.
\end{remark}

\section{Examples}
\label{sec:5}

We provide here some illustrative examples.

\subsection*{Example 1}

We consider the tridiagonal $12 \times 12$ Toeplitz matrix
\begin{equation}
A =T_{}(s,d,t), \qquad s=\frac{-1+\i}{10}, \quad d=\frac{-3+4 \i}{10}, \quad t = 2+\i.
\label{eq:example1}
\end{equation}
\begin{figure}[ht]
\centerline{
\includegraphics[scale=0.35,trim= 0mm 0.01mm 0mm 0mm]{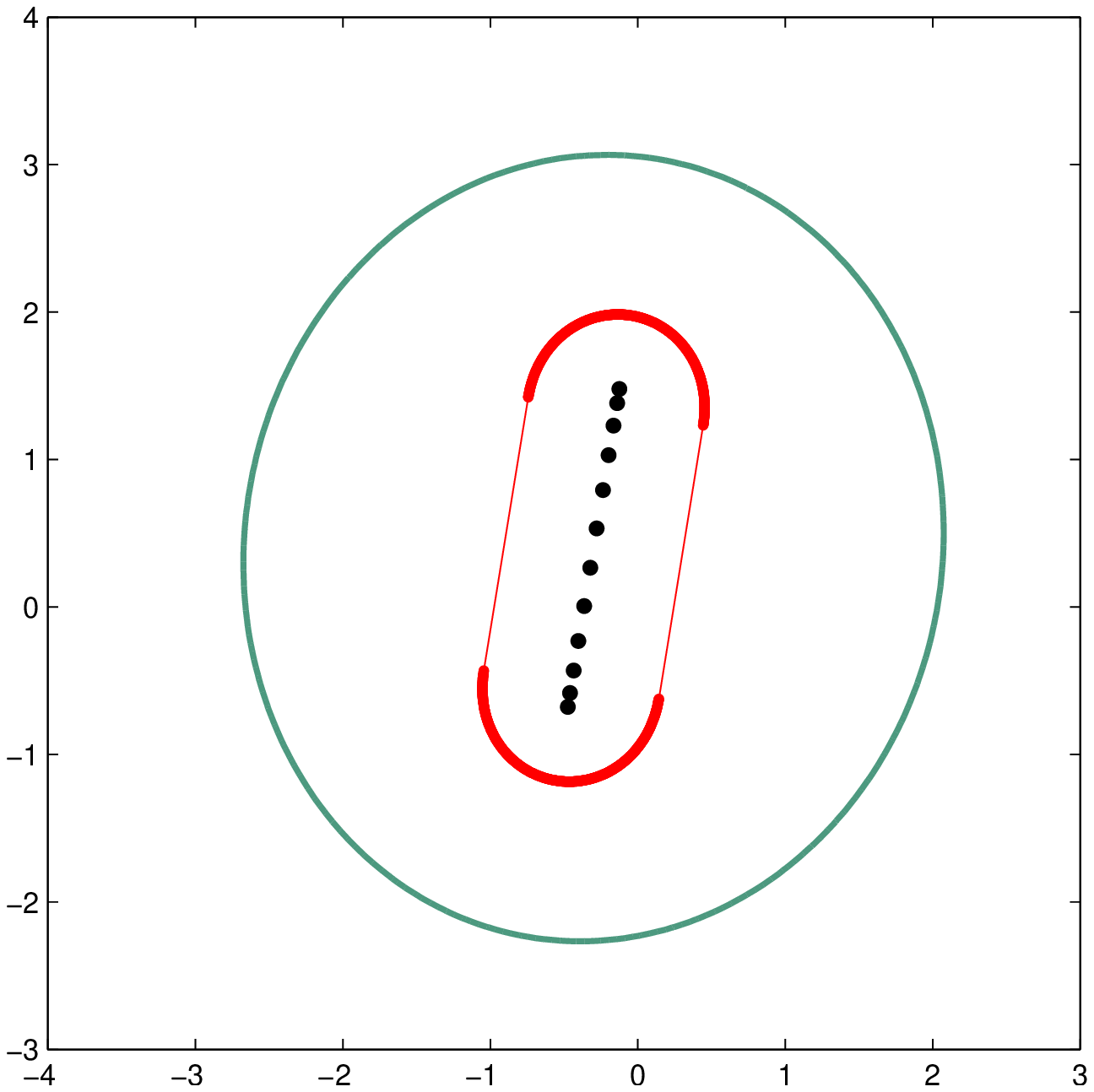}
\hskip 7mm
\includegraphics[scale=0.35,trim= 0mm 0.01mm 0mm 0mm]{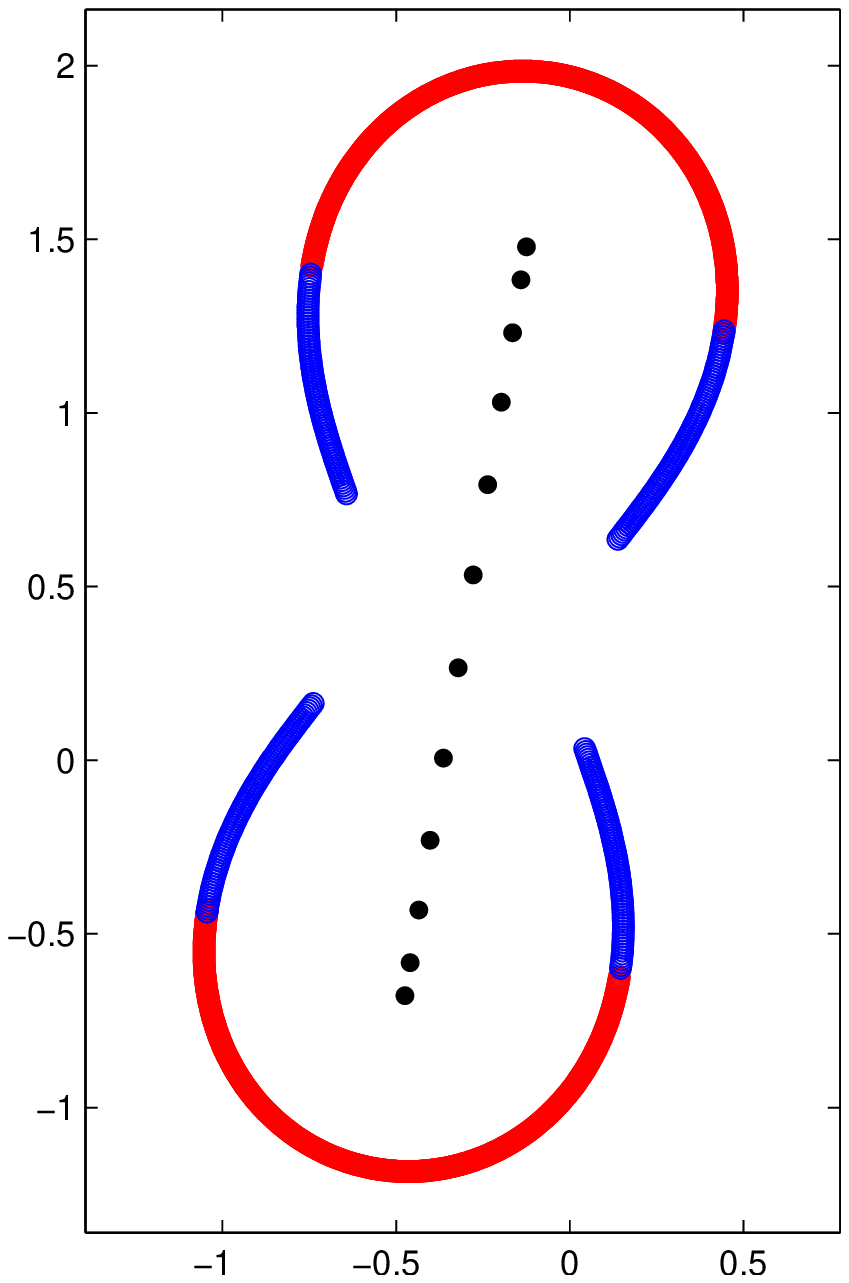}
\hskip 7mm
\includegraphics[scale=0.35,trim= 0mm 0.01mm 0mm 0mm]{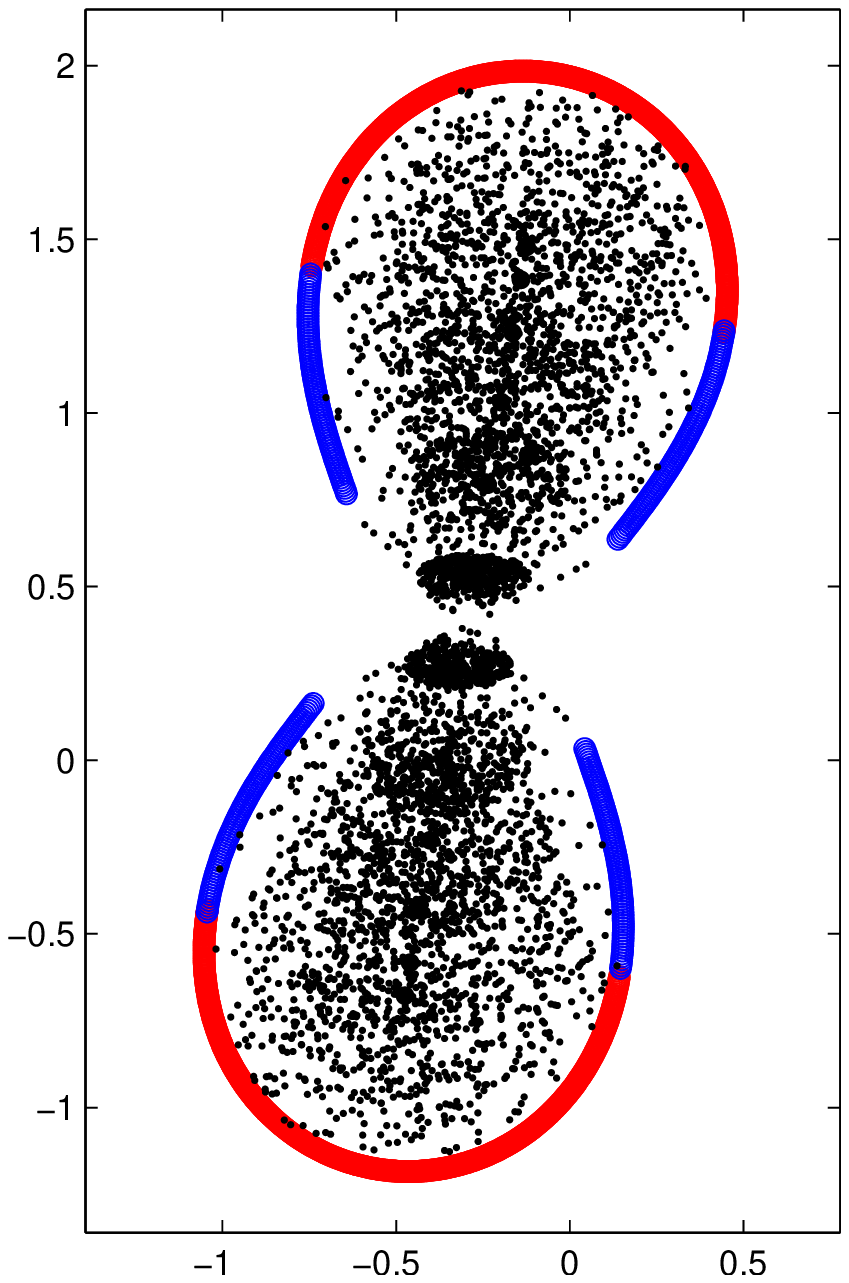}
}
\caption{Left picture: boundary of the unstructured pseudospectrum (gray oval) and section of the boundary of the structured $\eps$-pseudospectrum (in red), with $\eps=0.5$, of matrix (\ref{eq:example1}). Thin lines provide the convex hull of $\Lameps^{\mc T}(A)$.  Middle picture: further section (in blue) of $\partial \Lameps^{\mc T}(A)$. Right picture: black points are the spectra of $1000$ perturbed matrices obtained by adding the nominal matrix (\ref{eq:example1}) random perturbations with Gaussian distributed entries of Frobenius norm $\eps$. 
\label{fig:1}}
\end{figure}
In Figure \ref{fig:1}-left we plot the unstructured $\eps$-pseudospectrum and the computed section of the structured $\eps$-pseudospectrum for $\eps=0.5$, together with its convex hull (red points are boundary points computed by the rotated variant of Algorithm 1 discussed in Section \ref{subs:rot}, thin red lines give the convex hull). 

\begin{table}[ht]
\begin{center}
\begin{tabular}{l|l|l|}\hline
 $k$ & $\Re(\lambda_k)$ & $\alpha^{\mc T}_\eps(A) - \Re\left( \lambda_k \right)$ \\
\hline 
$0$ &  $-0.12508076372412$ & $0.59169457927624$ \\
$1$ &  $ 0.{\bf 4}1270494888923$ & $0.04056799024007$ \\                    
$3$ &  $ 0.{\bf 453}01543968544$ & $0.00025749944385$ \\                    
$5$ &  $ 0.{\bf 45327}100375008$ & $0.00000193537922$ \\                    
$7$ &  $ 0.{\bf 4532729}2456844$ & $0.00000001456086$ \\                    
$9$ &  $ 0.{\bf 453272939}01974$ & $0.00000000010956$ \\                    
$14$ & $ 0.{\bf 45327293912930}$ & $< 10^{-15}$
\\[0.3cm]
 \hline
\end{tabular}
\vspace{2mm}
\caption{Iterates and errors of Algorithm 1 applied to (\ref{eq:example1}) with $\eps=0.5$.}\label{tab:psa}
\end{center}
\end{table}

The behavior of Algorithm 1 is shown in Table \ref{tab:psa} and in Figure \ref{fig:2}-right 
where the iterates rapidly converge to the rightmost point. The estimated linear convergence rate is $r \approx 0.085$. 
In Figure \ref{fig:1}-middle we plot a further section of $\Lameps^{\mc T}(A)$
in the following way. Using the simple property 
\[
\Lameps^{\mc T}(A - \mu I) = \Lameps^{\mc T}(A) - \mu = \{ z = \lambda-\mu \colon \lambda \in \Lameps^{\mc T}(A) \}\;,
\]
we are able to use a variant of Algorithm 2 which converges to the point of minimal modulus of the
$\eps$-pseudospectrum of $A-\mu I$, being $\mu \in \bb C$ a point external to $\Lameps^{\mc T}(A)$.
The obtained value $\lambda_{\min}(\mu)$ is then shifted by $\mu$ and gives a point on the boundary of the
$\eps$-pseudospectrum.

\begin{figure}[ht]
\centerline{
\includegraphics[scale=0.35,trim= 0mm 0.01mm 0mm 0mm]{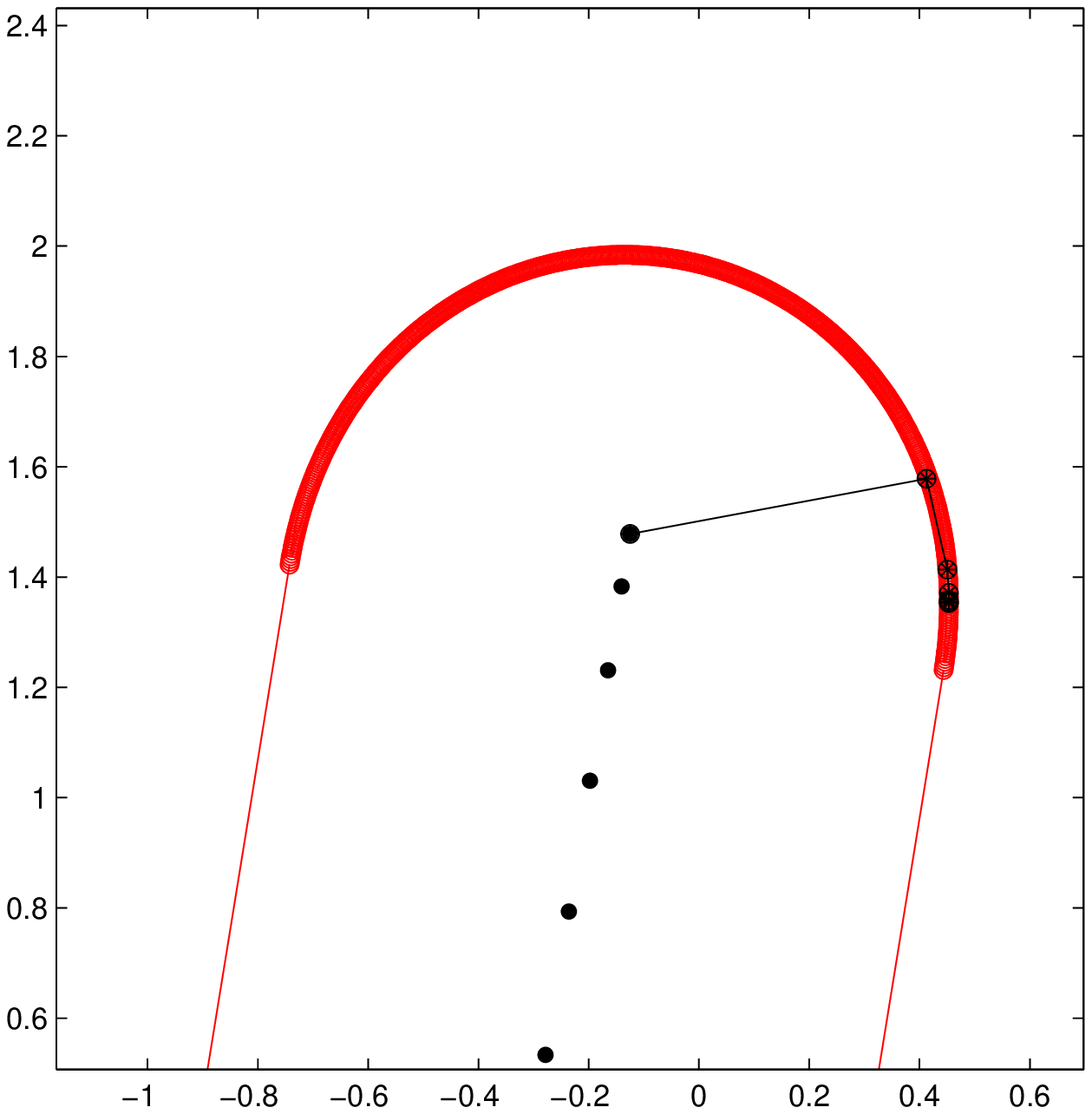}
\hskip 10mm
\includegraphics[scale=0.35,trim= 0mm 0.01mm 0mm 0mm]{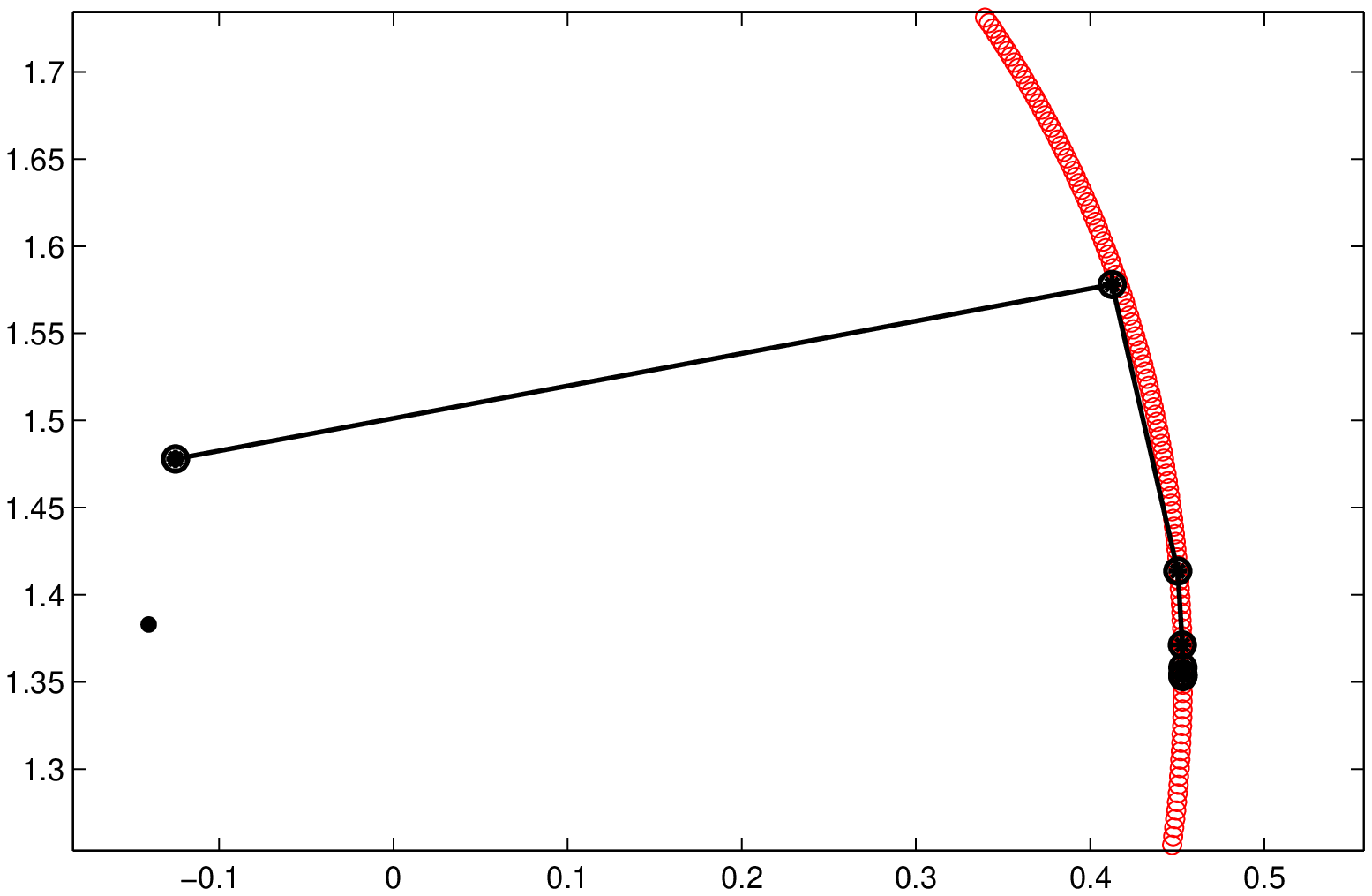}
}
\caption{In the left picture the zoom of the computed structured $\eps$-pseudospectrum. Right picture: 
iterates of Algorithm 1 converging to the rightmost point. \label{fig:2}}
\end{figure}

\subsection*{Example 2}
We consider the $30 \times 30$ pentadiagonal matrix 
\begin{equation}
A = T(0,10/19,0,0,10/19)
\label{eq:example2}
\end{equation}
generated by the symbol $\alpha(t)=\frac{10}{19} \left( t + t^{-2} \right)$.

\begin{figure}[h]
\centerline{
\includegraphics[scale=0.35,trim= 0mm 0.01mm 0mm 0mm]{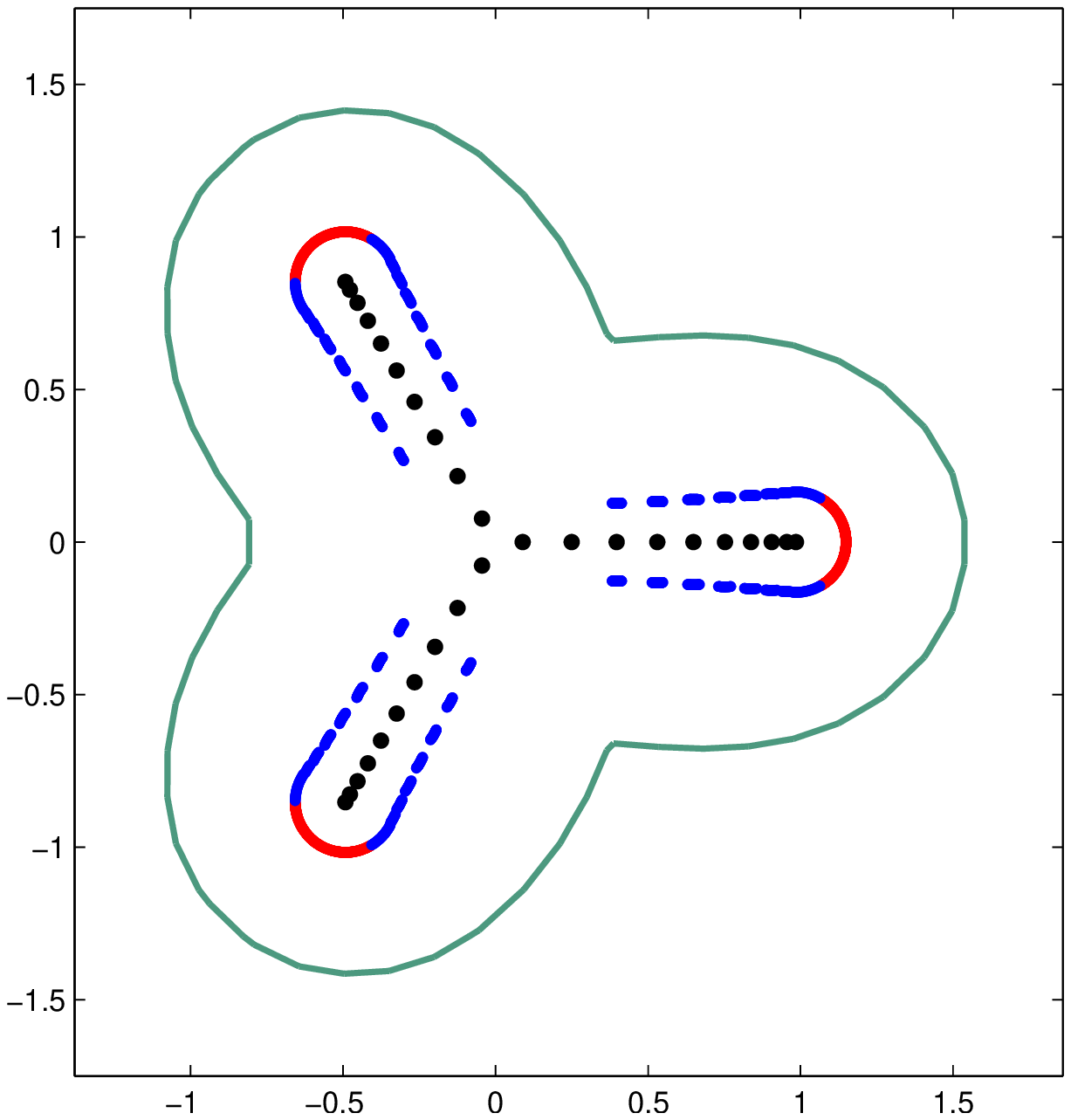}
\hskip 10mm
\includegraphics[scale=0.35,trim= 0mm 0.01mm 0mm 0mm]{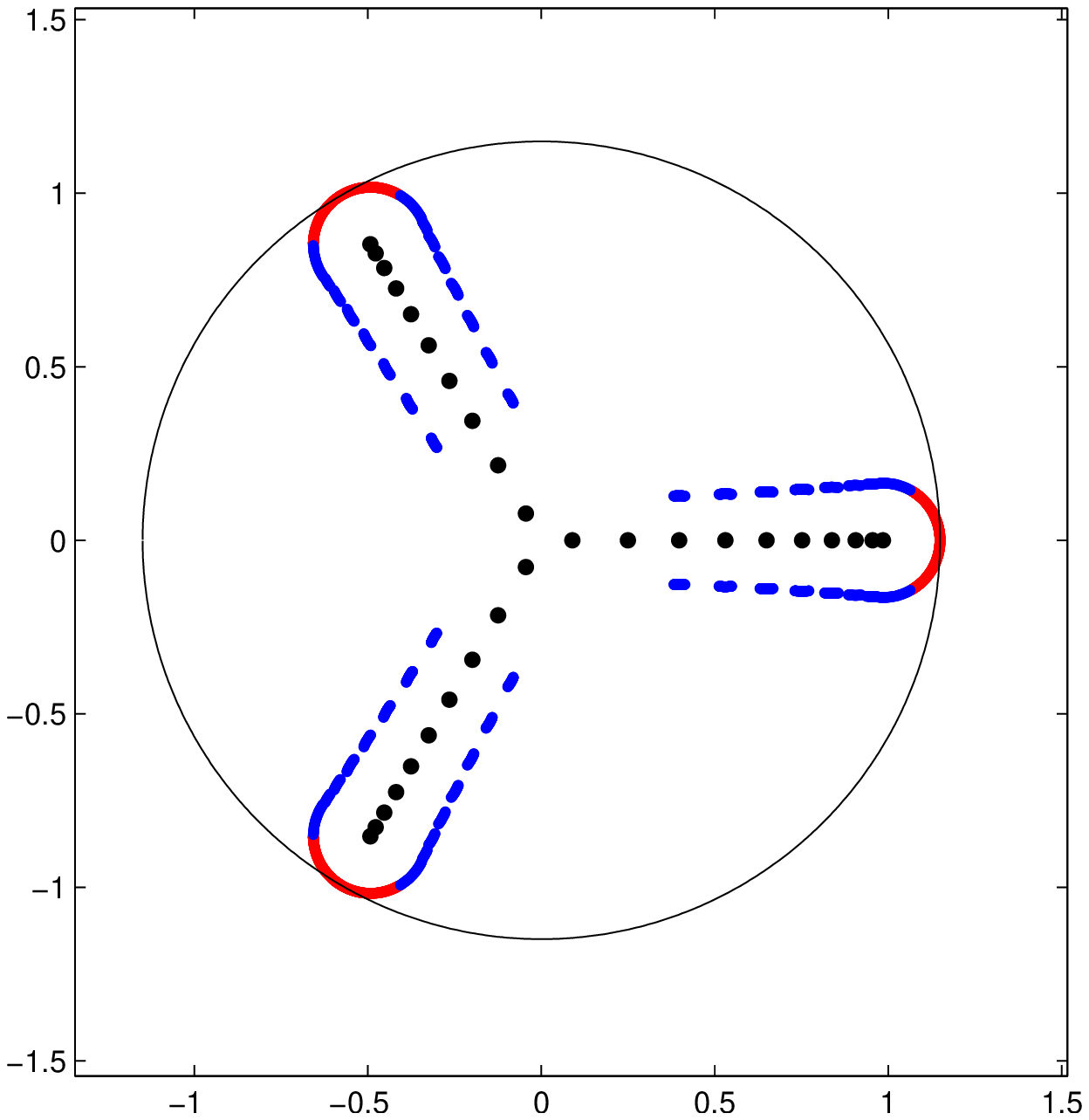}
}
\caption{Left picture: the unstructured $\eps$-pseudo\-spec\-trum of matrix (\ref{eq:example2}) is drawn with the section of the computed structured $\eps$-pseudospectrum for $\eps=0.5$. Right picture: the structured $\eps$-pseudo\-spec\-trum with the circle of radius $\rhoeps^{\mc T}(A)$. 
\label{fig:3}}
\end{figure}

In Figure \ref{fig:3} we show both the structured and sections of the unstructured pseudospectra
(the drawn blue section is not continuous, due to the fact that the boundary has oscillations
and the algorithm is not able to compute the corresponding concave parts).

\begin{figure}[!h]
\centerline{
\includegraphics[scale=0.35,trim= 0mm 0.01mm 0mm 0mm]{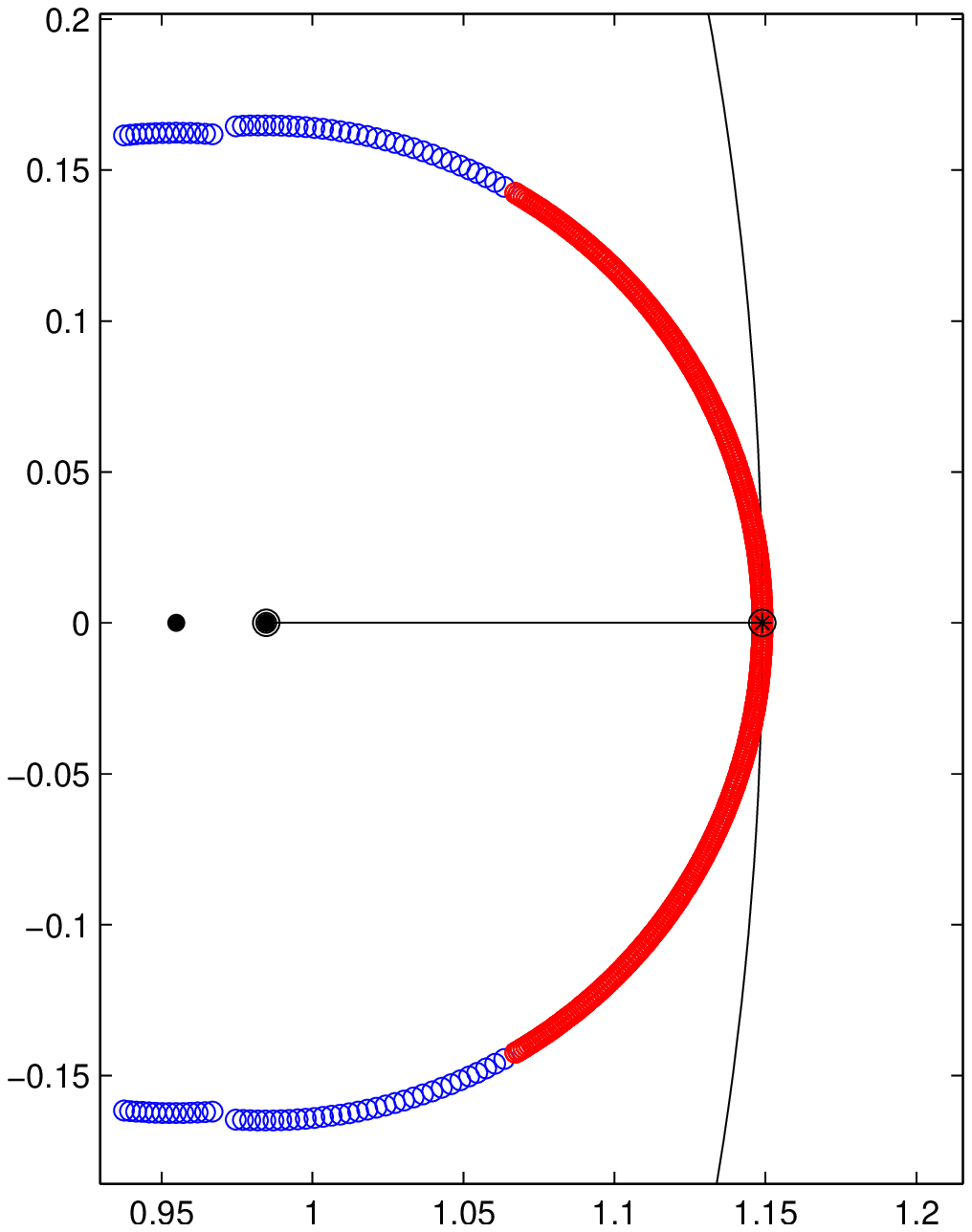}
\hskip 10mm
\includegraphics[scale=0.35,trim= 0mm 0.01mm 0mm 0mm]{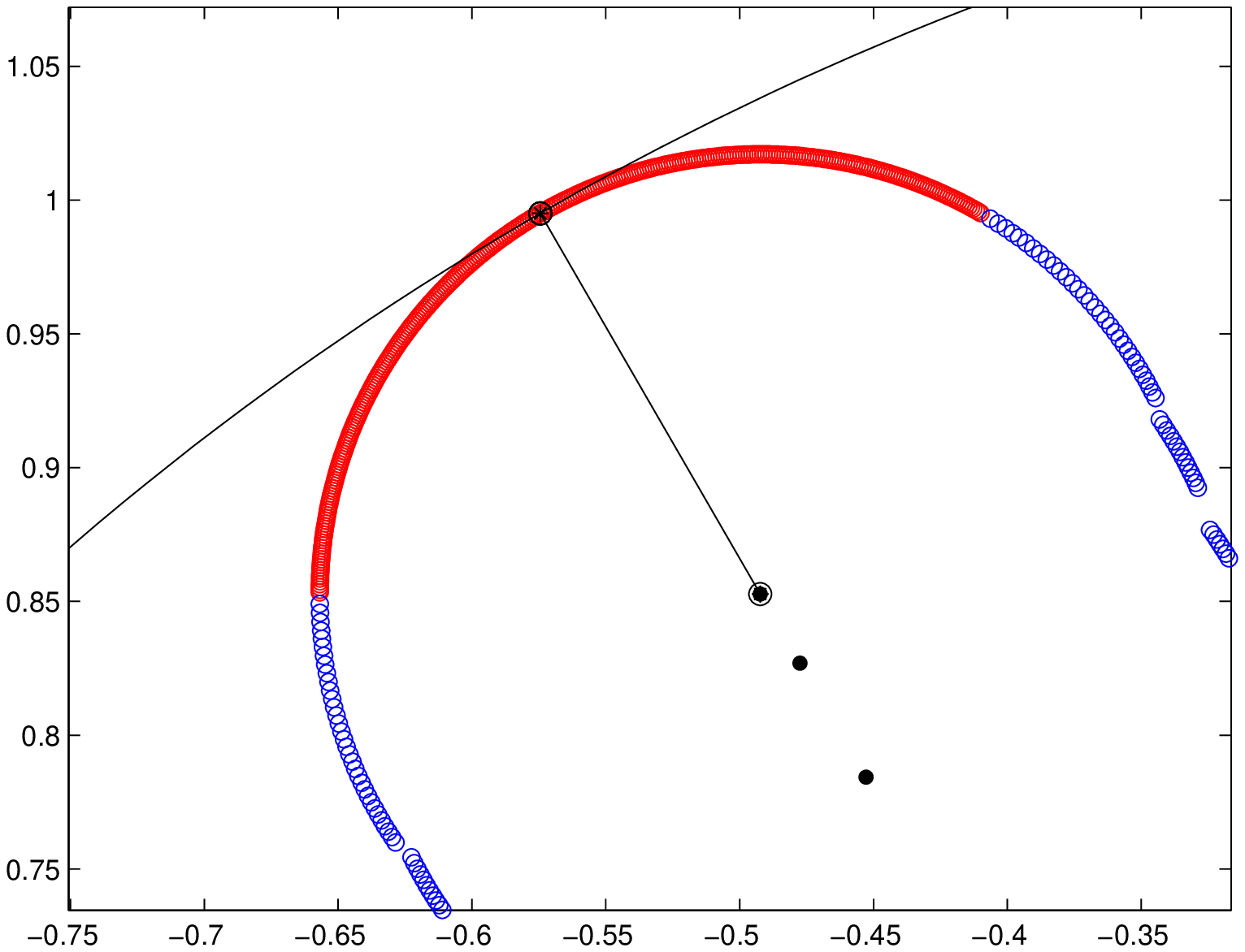}
}
\caption{Left picture: the iterates of Algorithm 1 applied to (\ref{eq:example2}) with $\eps=0.5$. 
Right picture: the iterates of Algorithm 2 applied to (\ref{eq:example2}) with $\eps=0.5$. \label{fig:4}}
\end{figure}

In Figure \ref{fig:4} we zoom the iterates generated by Algorithms 1 and 2 respectively to a rightmost point $\lambda_1$, that is $\Re(\lambda_1) = \aleps^{\mc T}(A)$ and to a point $\lambda_2$ of maximal modulus, that is $|\lambda_2| = \rhoeps^{\mc T}(A)$.

\subsection{Extension to Hankel matrices}

An extension to Hankel matrices is straightforward. We provide here an illustrative
example, complementary to Example 1.

\subsection*{Example 3}
We consider the anti-tridiagonal $12 \times 12$ Hankel matrix
\begin{equation}
A =H_{}(s,d,t), \qquad s=\frac{-1+\i}{10}, \quad d=\frac{-3+4 \i}{10}, \quad t = 2+\i
\label{eq:example3}
\end{equation}
that is the matrix with elements
\begin{eqnarray*}
&& a_{i,n+1-i} = d, \qquad i=1,\ldots,n
\\
&& a_{i,n-i} = s, \qquad i=1,\ldots,n-1
\\
&& a_{i+1,n+1-i} = t, \qquad i=1,\ldots,n-1.
\end{eqnarray*}
\begin{figure}[ht]
\centerline{
\includegraphics[scale=0.35,trim= 0mm 0.01mm 0mm 0mm]{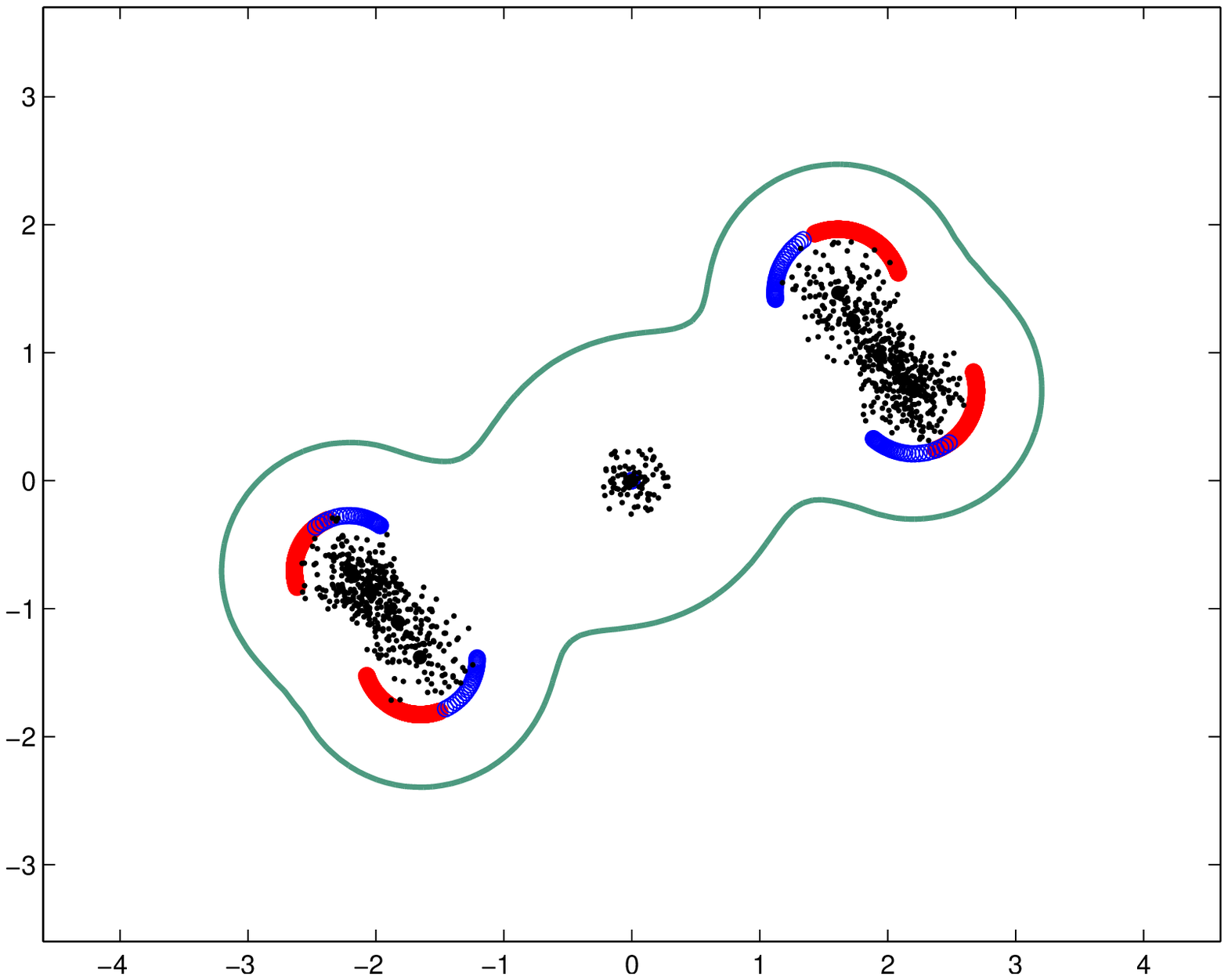}
\hskip 10mm
\includegraphics[scale=0.35,trim= 0mm 0.01mm 0mm 0mm]{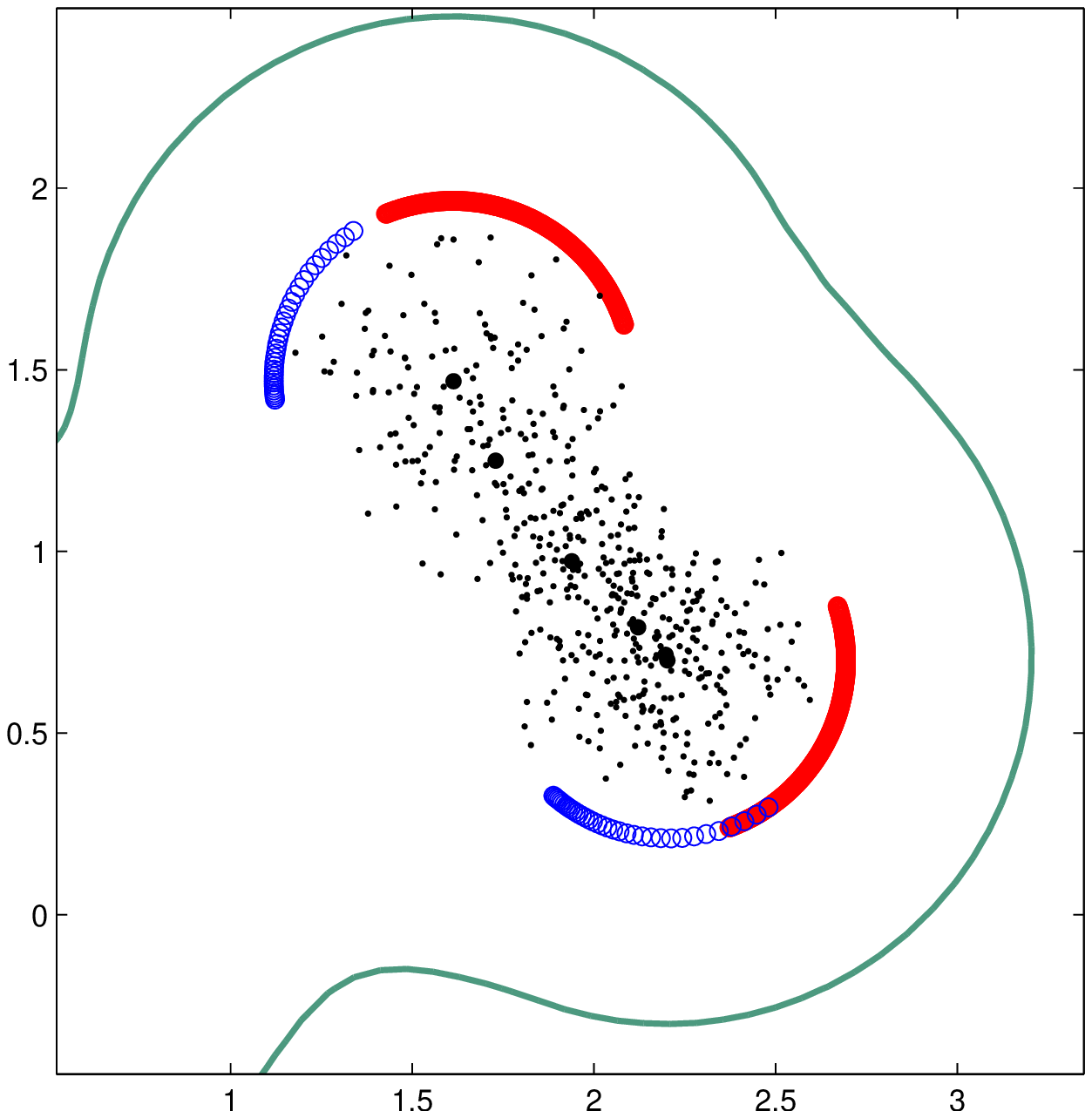}
}
\caption{Left picture: boundary of the the unstructured pseudospectrum (in gray) and section 
of the boundary of the structured $\eps$-pseudospectrum (in red and blue), with $\eps=1$, of the
Hankel matrix (\ref{eq:example3}). Black points are the spectra of $1000$ randomly selected perturbation 
matrices of norm $1$. 
Right picture: zoom. 
\label{fig:6}}
\end{figure}

The red section of the structured pseudospectrum is computed by the rotated implementation of
the basic algorithm to compute the pseudospectral abscissa. The blue section is computed by
a variant of the method to compute the pseudospectral radius. 

\subsection*{Acknowledgments}

We thank the Italian M.I.U.R. and G.N.C.S. for supporting this work.

\end{document}